\documentclass[letterpaper,11pt]{article}
\usepackage{fullpage}
\usepackage{amsmath,amssymb,amsthm,enumerate, hyperref} 
\usepackage[noadjust]{cite}
\usepackage{tikz}
\usetikzlibrary{positioning}
\usepackage{verbatim}
\usepackage[english]{babel}

\numberwithin{equation}{section} 

\theoremstyle{definition}
\newtheorem{definition}{Definition}[section]
\newtheorem{remark}[definition]{Remark}
\newtheorem{note}[definition]{Note}

\theoremstyle{plain}
\newtheorem{theorem}[definition]{Theorem}
\newtheorem{lemma}[definition]{Lemma}
\newtheorem{corollary}[definition]{Corollary}
\newtheorem{proposition}[definition]{Proposition}

\newcommand\scalemath[2]{\scalebox{#1}{\mbox{\ensuremath{\displaystyle #2}}}}

\newcommand{\gauss}[2]{\genfrac{[}{]}{0pt}{}{#1}{#2}}

\usepackage{scalerel}
\newsavebox{\foobox}
\newcommand{\slantbox}[2][0]{\mbox{%
        \sbox{\foobox}{#2}%
        \hskip\wd\foobox
        \pdfsave
        \pdfsetmatrix{1 0 #1 1}%
        \llap{\usebox{\foobox}}%
        \pdfrestore
}}
\newcommand\unslant[2][-.2]{%
  \mkern.5mu%
  \ThisStyle{\slantbox[#1]{$\SavedStyle#2$}}%
  \mkern-1mu%
}

\DeclareSymbolFont{sfletters}{OML}{iwonal}{m}{it}

\DeclareMathSymbol{\stheta}{\mathord}{sfletters}{"12}
\DeclareMathSymbol{\sphi}{\mathord}{sfletters}{"1E}
\DeclareMathSymbol{\svarphi}{\mathord}{sfletters}{"27}
\DeclareMathSymbol{\salpha}{\mathord}{sfletters}{"0B}
\DeclareMathSymbol{\sbeta}{\mathord}{sfletters}{"0C}
\DeclareMathSymbol{\sgamma}{\mathord}{sfletters}{"0D}
\DeclareMathSymbol{\stau}{\mathord}{sfletters}{"1C}
\DeclareMathSymbol{\sa}{\mathord}{sfletters}{"1A}

\newcommand{\mytheta}{\unslant\stheta}
\newcommand{\myphi}{\unslant\sphi}
\newcommand{\myvarphi}{\unslant\svarphi}

\DeclareSymbolFont{bsfletters}{OT1}{cmsmf}{bx}{n}
\DeclareMathSymbol{\Wsf}{\mathord}{bsfletters}{"57}

\def\sfa{\mathsf a}
\def\sfb{\mathsf b}
\def\sfc{\mathsf c}
\def\sfd{\mathsf d}

\def\sfr{\mathsf r}

\begin{document}

\title{\bf Grassmann graphs, degenerate DAHA, and non-symmetric dual $q$-Hahn polynomials}

\author {Jae-Ho Lee\thanks{Department of Mathematics and Statistics, University of North Florida, Jacksonville, FL 32224, U.S.A, E-mail: \texttt{jaeho.lee@unf.edu}} 
}
\date{}

\maketitle

\begin{abstract}
We discuss the Grassmann graph $J_q(N,D)$  with $N \geq 2D$, having as vertices the $D$-dimensional subspaces of an $N$-dimensional vector space over the finite field $\mathbb{F}_q$.
This graph is distance-regular with diameter $D$; to avoid trivialities we assume $D\geq 3$.
Fix a pair of a Delsarte clique $C$ of $J_q(N,D)$ and a vertex $x$ in $C$.
We construct a $2D$-dimensional irreducible module $\mathbf{W}$ for the Terwilliger algebra $\mathbf{T}$ of $J_q(N,D)$ associated with the pair $x$, $C$. 
We show that $\mathbf{W}$ is an irreducible module for the confluent Cherednik algebra $\mathcal{H}_\mathrm{V}$ and describe how the $\mathbf{T}$-action on $\mathbf{W}$ is related to the $\mathcal{H}_\mathrm{V}$-action on $\mathbf{W}$.
Using the $\mathcal{H}_\mathrm{V}$-module $\mathbf{W}$, we define non-symmetric dual $q$-Hahn polynomials and prove their recurrence and orthogonality relations from a combinatorial viewpoint.


\end{abstract}

\section{Introduction}

In this paper, we continue to develop the link between the theory of  $Q$-polynomial distance-regular graphs and the theory of double affine Hecke algebras (DAHAs); cf. \cite{1984BanIto, 1973Delsarte, 1992Cherednik, 2013LeeLAA, 2017LeeJCTA, 2018LeeTanakaSIGMA}.
We briefly summarize our results concerning the link.  
In \cite{2013LeeLAA}, we considered a $Q$-polynomial distance-regular graph  that corresponds to $q$-Racah polynomials, at the top level (i.e. $_{4}\phi_{3}$) in the terminating branch of the $q$-Askey scheme \cite{2010KLS}.
Assuming that the graph contains a clique with maximal possible size (i.e. Delsarte clique), we introduced the generalized Terwilliger algebra $\mathbf{T}(x,C)$, which is a non-commutative semisimple matrix $\mathbb{C}$-algebra attached to every pair of a Delsarte clique $C$ and a vertex $x \in C$ of the graph.
We showed that each such pair $x,C$ gives rise to a vector space that has an irreducible module structure for both $\mathbf{T}(x,C)$ and a DAHA of type $(C^\vee_1, C_1)$, the most general  DAHA of rank one \cite{2004OblomkovIMRN}. 
In the following paper \cite{2017LeeJCTA}, we captured the non-symmetric $q$-Racah polynomials from that vector space, a discrete version of non-symmetric Askey-Wilson polynomials introduced by Sahi \cite{1999SahiAnnMath}, and gave a combinatorial interpretation for their orthogonality relations.

We note that, however, the results obtained in \cite{2013LeeLAA, 2017LeeJCTA} may remain at the purely algebraic level; because there is no known example of a (non-trivial) $Q$-polynomial distance-regular graph with large diameter (at least ten)\footnote{For small diameter, there are infinitely many examples
of bipartite $Q$-polynomial distance-regular graphs of $q$-Racah type, for which every edge is a Delsarte clique.}  that corresponds to $q$-Racah polynomials and contains a Delsarte clique.
To complement this shortcoming, in the subsequent paper \cite{2018LeeTanakaSIGMA} we dealt with the dual polar graphs as a concrete combinatorial example in the context of the theory developed in \cite{2013LeeLAA, 2017LeeJCTA}.
The dual polar graphs are a classical family of $Q$-polynomial distance-regular graphs and correspond to the dual $q$-Krawtchouk polynomials.
Applying techniques of \cite{2013LeeLAA, 2017LeeJCTA} to a dual polar graph, we obtained an irreducible module for a nil-DAHA\footnote{This nil-DAHA is isomorphic to the confluent Cherednik algebra $\mathcal{H}_{\mathrm{III}}$ that corresponds to Al-Salam-Chihara polynomials; cf. \cite{2018LeeTanakaSIGMA, 2014Mazzocco}.} of type $(C^\vee_1, C_1)$ \cite{2015CheOrrMathZ}, which is a specialization of the DAHA of type $(C^\vee_1, C_1)$.
We then captured the non-symmetric dual $q$-Krawtchouk polynomials, a discrete version of non-symmetric Al-Salam-Chihara polynomials \cite{2014Mazzocco}, from a nil-DAHA module. We also described their recurrence and orthogonality relations from a combinatorial point of view.

In the present paper, as another specific combinatorial object with strong regularity, we discuss the Grassmann graphs in the context of our study to develop the theory of  \cite{2013LeeLAA, 2017LeeJCTA} further.
The Grassmann graphs are a classical family of $Q$-polynomial distance-regular graphs and correspond to the dual $q$-Hahn polynomials which lie in between  $q$-Racah and dual $q$-Krawtchouk polynomials in the $q$-Askey scheme; see Figure \ref{q-Askey scheme}.
\begin{figure}
\centering
\scalemath{0.9}{
\begin{tikzpicture}[>=stealth,thick, every node/.style={shape=rectangle,rounded corners}, align=center, text width=16em, minimum height=1em]
    \node[draw, minimum height=2.7em] (c1) at (0,0) {\small $q$-Racah polynomials \\ ($Q$-polynomial distance-regular graphs)};
   \node[draw, minimum height=2.7em] (c3) at (0, -2) {\small Dual $q$-Hahn polynomials \\ (Grassmann graphs)};
    \node[draw, minimum height=2.7em] (c6) at (0, -4) {\small Dual $q$-Krawtchouk polynomials\\ (Dual polar graphs)};
    
    \node[draw, minimum height=2.7em] (d1) at (8.5,0) {\small The DAHA $\mathcal{H}$ of type $(C^\vee_1, C_1)$};
   \node[draw, minimum height=2.7em] (d2) at (8.5, -2) {\small The confluent Cherednik algebra $\mathcal{H}_\mathrm{V}$};
    \node[draw, minimum height=2.7em] (d3) at (8.5, -4) {\small The confluent Cherednik algebra $\mathcal{H}_\mathrm{III}$};

    \node (r1) [text width=1.5em] at (-4,0) {(${}_4\phi_3$)};
    \node (r2) [text width=1.5em] at (-4,-2) {(${}_3\phi_2$)};
    \node (r3) [text width=1.5em] at (-4, -4) {(${}_3\phi_2$)};
    
    \draw[->] (c1) to (c3);
    \draw[->] (c3) to (c6);
    \draw[->] (d1) to (d2);
    \draw[->] (d2) to (d3);
    
    \draw[<->, dashed, thick] (c1) to (d1);
    \draw[<->, dashed, thick] (c3) to (d2);
    \draw[<->, dashed, thick] (c6) to (d3);

\end{tikzpicture}}
\caption[Caption Footnoot]{Part of the $q$-Askey scheme and the corresponding (degenerate) DAHAs\footnotemark}\label{q-Askey scheme}
\end{figure}
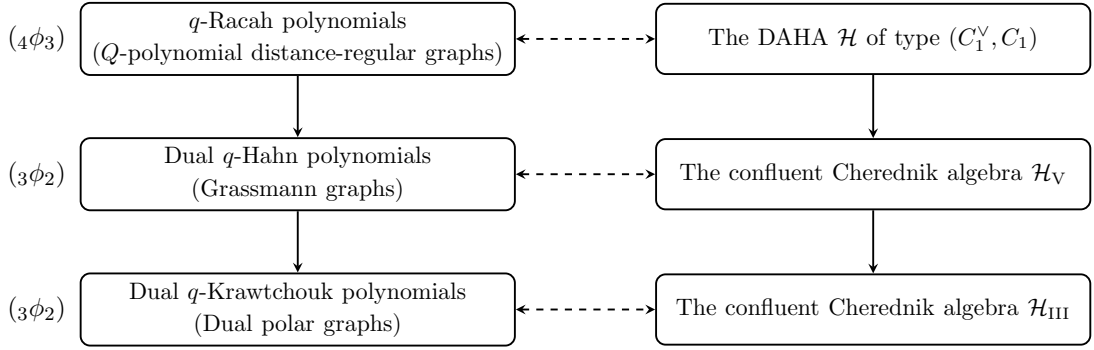
\footnotetext{Recently, the duality and its limit behavior of three families (Askey-Wilson, continuous dual $q$-Hahn, Al-Salam-Chihara) of the $q$-Askey scheme and the corresponding degenerate DAHAs were dealt with by Koornwinder and Mazzocco; cf. \cite{2018KoornMazzo}.}
The main results of this paper are as follows.
Let $J_q(N,D)$ denote a Grassmann graph, where $q$ is a prime power and $N$, $D$ are positive integers with $N \geq 2D$; see the definition in Section \ref{Section:Grassmann graphs}. 
Fix a pair of a Delsarte clique $C$ in $J_q(N,D)$ and a vertex $x$ in $C$.
Applying the methods used in \cite{2013LeeLAA, 2017LeeJCTA, 2018LeeTanakaSIGMA} to $J_q(N,D)$, we construct a $2D$-dimensional irreducible $\mathbf{T}(x,C)$-module $\mathbf{W}$ and show that $\mathbf{W}$ has a module structure for the \emph{confluent Cherednik algebra} $\mathcal{H}_\mathrm{V}$ \cite{2016Mazzocco}; cf. Definition \ref{Def:CheAlgHv}.
We then demonstrate how the $\mathbf{T}(x,C)$-action on $\mathbf{W}$ is related to the $\mathcal{H}_\mathrm{V}$-action on $\mathbf{W}$; cf. Theorems \ref{thm:1stResult} and \ref{thm:2ndResult}.
From the $\mathcal{H}_\mathrm{V}$-module $\mathbf{W}$, we obtain \emph{non-symmetric dual $q$-Hahn polynomials}, a discrete version of non-symmetric continuous dual $q$-Hahn polynomials \cite[Section 2]{2014Mazzocco}, and describe their recurrence and orthogonality relations; cf. Theorems \ref{thm:3rdResult} and \ref{thm:4thResult}. 
We should mention that all the formulas in the present paper are described in terms of the scalars $q$, $N$, and $D$, not depending on our fixed pair $x, C$.

This paper is organized as follows.
In Section \ref{Section:DRGs} we recall some preliminaries concerning $Q$-polynomial distance-regular graphs and the Terwilliger algebra.
In Section \ref{Section:Grassmann graphs} we discuss a Grassmann graph $\Gamma=J_q(N,D)$ with diameter $D\geq 3$ and its properties. 
We also discuss the Terwilliger algebra of $\Gamma$ associated with a Delsarte clique.
In Section \ref{Section:algebraT} we fix a Delsarte clique $C$ and a vertex $x$ in $C$. 
We discuss the generalized Terwilliger algebra $\mathbf{T}=\mathbf{T}(x,C)$ of $\Gamma$ and construct the so-called primary $\mathbf{T}$-module $\mathbf{W}$.
In Section \ref{Section:LS of dual q-Hahn} we discuss the theory of Leonard systems. In particular, we treat a family of Leonard systems that corresponds to dual $q$-Hahn polynomials.
In Section \ref{Section:primary T-module} we deal with four dual $q$-Hahn Leonard systems obtained from $\mathbf{W}$. 
In Section \ref{Section:Hv} we discuss the confluent Cherednik algebra $\mathcal{H}_\mathrm{V}$ and construct a $\mathbb{C}$-algebra homomorphism from $\mathcal{H}_\mathrm{V}$ to $\mathrm{End}(\mathbf{W})$, which gives an $\mathcal{H}_\mathrm{V}$-module structure on $\mathbf{W}$.
We discuss a relationship between the action of $\mathcal{H}_\mathrm{V}$ and  the action of $\mathbf{T}$ on $\mathbf{W}$.
We specialize the DAHA of type $(C^\vee_1, C_1)$ to get a nil-DAHA $\overline{\mathcal{H}}$ and discuss how $\overline{\mathcal{H}}$ is related to $\mathcal{H}_\mathrm{V}$.
In Section \ref{Section:nonsym dual q-Hahn poly} we introduce non-symmetric dual $q$-Hahn polynomials $\ell^\pm_i$ and give a combinational interpretation for $\ell^\pm_i$.
In Section \ref{Section:rec orth relations} we deal with recurrence relations and orthogonality relations for $\ell^\pm_i$.

Throughout this paper, we use the following notation.
For a non-empty finite set $X$, let $\mathrm{Mat}_X(\mathbb{C})$ denote the $\mathbb{C}$-algebra consisting of the complex square matrices indexed by $X$.
Let $\mathbb{C}^X$ denote the $\mathbb{C}$-vector space consisting of the complex column vectors indexed by $X$.
We endow $\mathbb{C}^X$ with the Hermitian inner product $\langle \cdot , \cdot \rangle = \langle \cdot , \cdot \rangle_{\mathbb{C}^X}$ which satisfies $\langle u, v \rangle = u^t\bar{v}$ for $u,v \in \mathbb{C}^X$, where $t$ denotes transpose and $\bar{ \ }$ denotes complex conjugate.
Abbreviate $\lVert u \rVert^2=\langle u, u \rangle$ for all $u \in \mathbb{C}^X$.
For $y \in X$, let $\hat{y}$ denote the vector in $\mathbb{C}^X$ with a $1$ in the $y$-coordinate and $0$ in all other coordinates.
For a subset $Y \subseteq X$, define $\hat{Y} = \sum_{y \in Y}\hat{y}$, called the \textit{characteristic vector} of $Y$.
Let $\mathbb{C}[\zeta, \zeta^{-1}]$ denote the space of Laurent polynomials in one variable $\zeta$.
A Laurent polynomial $f(\zeta)$ is said to be \emph{symmetric} if $f(\zeta)=f(\zeta^{-1})$, and \emph{non-symmetric} otherwise.
We view symmetric Laurent polynomials as ordinary polynomials in the variable $\lambda :=\zeta+\zeta^{-1}$.
Assume that $q \in \mathbb{C}^*$ is not a root of unity.
For $\alpha \in \mathbb{C}$,
\begin{equation}
	(\alpha;q)_0 :=1 \quad  \text{and} \quad (\alpha;q)_n := (1-\alpha)(1-\alpha q) \cdots (1-\alpha q^{n-1}), \quad n=1,2,3,\ldots.
\end{equation}
For $\alpha_1, \alpha_2, \ldots, \alpha_{s+1}, \beta_1, \beta_2, \ldots, \beta_{s} \in \mathbb{C}$,
\begin{equation}
	{}_{s+1}\phi_s \left( \left. \begin{array}{c} \alpha_1, \alpha_2, \ldots, \alpha_{s+1} \\ \beta_1, \beta_2,  \ldots, \beta_s \end{array}\right| q; \zeta \right) = \sum^\infty_{n=0}\frac{(\alpha_1;q)_n(\alpha_2;q)_n \cdots (\alpha_{s+1};q)_n}{(\beta_1;q)_n(\beta_2;q)_n\cdots (\beta_s;q)_n}\frac{\zeta^n}{(q;q)_n}.
\end{equation}
For integers $0 \leq m \leq n$, we denote the Gaussian binomial coefficient by
\begin{equation}
	\gauss{n}{m} = \gauss{n}{m}_q = \frac{(q^n-1)(q^{n-1}-1)\cdots(q^{n-m+1}-1)}{(q^m-1)(q^{m-1}-1)\cdots(q-1)}.
\end{equation}
We remark that if $q$ is set to a prime power then $\gauss{n}{m}$ is equal to the number of $m$-dimensional subspaces of an $n$-dimensional vector space over a finite field $\mathbb{F}_q$.
In what follows, we assume that $q$ is a prime power unless otherwise stated.

\section{Preliminaries: Distance-regular graphs}\label{Section:DRGs}

In this preliminary section, we recall some basic aspects of distance-regular graphs that we need later in the paper. 
Let $\Gamma$ be a connected simple graph with finite vertex set $X$ and diameter $D\geq 3$.
For a vertex $x$ in $X$, define
\begin{equation}\label{dist.part(x)}
	\Gamma_i(x) := \{ y \in X : \partial(x,y) = i \}, \quad 0 \leq i \leq D,
\end{equation}
where $\partial(x,y)$ is the shortest path-length distance function between $x$ and $y$.
We abbreviate $\Gamma(x)=\Gamma_1(x)$.
For an integer $k$, $\Gamma$ is said to be $k$-\emph{regular} (or \emph{regular with valency $k$}) whenever $|\Gamma(x)|=k$ for all $x$ in $X$. 
We say that $\Gamma$ is \emph{distance-regular} whenever for every $i$, $0 \leq i \leq D$, and for every pair of vertices $x$, $y$ in $X$ with $\partial(x,y)=i$, there are constant numbers $a_i$, $b_i$, $c_i$  such that 
\begin{equation}\label{IntNum.Gamma}
	c_i = | \Gamma_{i-1}(x) \cap \Gamma(y) |, \qquad
	a_i = | \Gamma_{i}(x) \cap \Gamma(y) |, \qquad
	b_i = | \Gamma_{i+1}(x) \cap \Gamma(y) |,
\end{equation}
where $\Gamma_{-1}(x)$ and $\Gamma_{D+1}(x)$ are empty sets.
Observe that $c_0=b_D=0$, $b_{i-1}c_i \ne 0$, $1\leq i \leq D$, and $c_1=1$.
Observe also that $\Gamma$ is $b_0$-regular and $a_i+b_i+c_i=b_0$ for $0 \leq i \leq D$.
The constants $a_i$, $b_i$, $c_i$ are called the \emph{intersection numbers} of $\Gamma$.

Assume that $\Gamma$ is distance-regular.
For $0 \leq i \leq D$, define the matrix $A_i$ in $\mathrm{Mat}_X(\mathbb{C})$ such that $(x,y)$-entry of $A_i$ is $1$ if $\partial(x,y)=i$ and $0$ otherwise.
We call $A_i$ the \emph{$i$-th distance matrix} of $\Gamma$.
Observe that $A_0=I$, the identity matrix in $\mathrm{Mat}_X(\mathbb{C})$.
We abbreviate $A=A_1$ and call this the \emph{adjacency matrix} of $\Gamma$.
The \emph{Bose-Mesner algebra} of $\Gamma$ is the (commutative) semisimple subalgebra $M$ of $\mathrm{Mat}_X(\mathbb{C})$ generated by $I, A, A_2,\ldots,A_D$. 
Observe that
\begin{equation*}
	AA_i = b_{i-1}A_{i-1}+a_iA_i+c_{i+1}A_{i+1}, \qquad 0 \leq i \leq D,
\end{equation*}
where we set $b_{-1}A_{-1}=0$ and $c_{D+1}A_{D+1}=0$. From this recurrence, it follows that  there is a polynomial $v_i \in \mathbb{C}[\lambda]$ of degree $i$ such that $v_i(A)=A_i$ for $0 \leq i \leq D$.
It follows that $A$ generates $M$, and that the matrices $A_i$, $0 \leq i \leq D$, form a basis for $M$.
Since $A$ is real symmetric and generates $M$, it has $D+1$ mutually distinct real eigenvalues $\theta_0, \theta_1, \ldots, \theta_D$.
We always set $\theta_0:=b_0$.
For $0 \leq i \leq D$, let $E_i \in \mathrm{Mat}_X(\mathbb{C})$ be the orthogonal projection onto the eigenspace of $\theta_i$.
Observe that $E_iE_j = \delta_{ij}E_i$, $0 \leq i, j \leq D$, and $\sum^D_{i=0}E_i=I$.
We have
\begin{equation*}
	A = \sum^D_{i=0}\theta_i E_i,
\end{equation*}
so that the matrices $E_i$, $0 \leq i \leq D$, form another basis for $M$.

We recall the $Q$-polynomial property of $\Gamma$.
The Bose-Mesner algebra $M$ of $\Gamma$ is closed under entrywise multiplication, denoted by $\circ$, since $A_i\circ A_j = \delta_{ij}A_i$, $0 \leq i,j \leq D$.
We say that $\Gamma$ is \emph{$Q$-polynomial} with respect to the ordering $E_0, E_1, \ldots, E_D$ (or $\theta_0, \theta_1, \ldots, \theta_D$) if there are scalars $a^*_i$, $b^*_i$, $c^*_i$, $0 \leq i \leq D$, such that $b^*_D=c^*_0=0$, and $b^*_{i-1}c^*_i \ne 0$ for $1\leq i \leq D$, and
\begin{equation*}
	|X| (E_1\circ E_i) = b^*_{i-1}E_{i-1} + a^*_iE_i+c^*_{i+1}E_{i+1}, \qquad 0 \leq i \leq D,
\end{equation*}
where we set $b^*_{-1}E_{-1}=0$ and $c^*_{D+1}E_{D+1}=0$.
From this recurrence, it follows that there is a polynomial $v^*_i \in \mathbb{C}[\lambda]$ of degree $i$ such that $v^*_i(E_1) = E_i$ for $0 \leq i \leq D$, where the multiplication is under $\circ$.
Write $E_1=|X|^{-1}\sum^D_{i=0}\theta^*_iA_i$. Then the scalars $\theta^*_i$, $0\leq i \leq D$, are real and mutually distinct. 
We note that $\theta^*_0 = \mathrm{trace}(E_1) = \mathrm{rank}(E_1)$.

Assume that $\Gamma$ is $Q$-polynomial with respect to the ordering $E_0, E_1, \ldots, E_D$.
Fix a vertex $x$ in $X$.
For $0 \leq i \leq D$, let $E^*_i=E^*_i(x)$ denote the diagonal matrix in $\mathrm{Mat}_X(\mathbb{C})$ with $(y,y)$-entry $1$ if $\partial(x,y)=i$ and $0$ otherwise, i.e., $E^*_i = \mathrm{diag}(A_i\hat{x})$.
Observe that $E^*_iE^*_j=\delta_{ij}E^*_i$, $0\leq i,j \leq D$, and $\sum^D_{i=0}E^*_i=I$.
The \emph{dual Bose-Mesner algebra} of $\Gamma$ with respect to $x$ is the (commutative) semisimple subalgebra $M^*=M^*(x)$ of $\mathrm{Mat}_X(\mathbb{C})$ generated by $E^*_0, E^*_1, \ldots, E^*_D$.
Note that the matrices $E^*_i$, $0 \leq i \leq D$, form a basis for $M^*$.
Let  $A^*=A^*(x)$ denote the diagonal matrix in $\mathrm{Mat}_X(\mathbb{C})$ with $(y,y)$-entry $(|X|E_1)_{xy}$ for $y\in X$, i.e., $A^* = \mathrm{diag}(|X|E_1\hat{x})$.
Then
\begin{equation}\label{A*;linComb(E*i)}
	A^* = \sum^D_{i=0}\theta^*_iE^*_i,
\end{equation}
from which it follows that $A^*$ generates $M^*$. We call $A^*$ the \emph{dual adjacency matrix} of $\Gamma$ with respect to $x$.
Observe that the scalars $\theta^*_i$ are the eigenvalues of $A^*$, called the \emph{dual eigenvalues} of $\Gamma$.
The \emph{Terwilliger algebra} (or \emph{subconstituent algebra}) of $\Gamma$ with respect to $x$ is the subalgebra $T=T(x)$ of $\mathrm{Mat}_X(\mathbb{C})$ generated by $M$, $M^*$ \cite{1992TerJACO(1), 1993TerJACO(2), 1993TerJACO(3)}.
Note that the matrices $A$, $A^*$ generate $T$.
Note also that $T$ is (non-commutative) semisimple and any two non-isomorphic irreducible $T$-modules in $\mathbb{C}^X$ are orthogonal.
The following are relations in $T$:
\begin{equation*}
	E^*_iAE^*_j = 0, \qquad  E_iA^*E_j=0 \qquad \text{if} \quad |i-j|>1,
\end{equation*}
for $0 \leq i, j \leq D$; cf. \cite[Lemma 3.2]{1992TerJACO(1)}.

We observe that the subspace $M\hat{x}$ of $\mathbb{C}^X$ has bases $\{A_i\hat{x}\}^D_{i=0}$ and $\{E_i\hat{x}\}^D_{i=0}$, and that $A_i\hat{x}=E^*_i\hat{X}$, $0 \leq i \leq D$.
It follows that $M\hat{x}$ is same as the subspace $M^*\hat{X}$ of $\mathbb{C}^X$, and therefore $M\hat{x}$ is an irreducible $T$-module, called the \emph{primary} $T$-module.
The actions of $A$, $A^*$ on $M\hat{x}$ are given as follows: for $0 \leq i \leq D$,
\begin{equation*}
	A.A_i\hat{x} = b_{i-1}A_{i-1}\hat{x}+a_iA_i\hat{x}+c_{i+1}A_{i+1}\hat{x}, \qquad   A^*.A_i\hat{x} = \theta^*_iA_i\hat{x}.
\end{equation*}
For more information regarding distance-regular graphs, we refer to \cite{1984BanIto, 1989BCN, 2016DKT}.

\section{Grassmann graphs}\label{Section:Grassmann graphs}

Recall $q$ a prime power. Let $N, D$ be positive integers with the restriction $N \geq 2D$.
Let $V$ be an $N$-dimensional vector space over a finite field $\mathbb{F}_q$, and let $X$ be the collection of $D$-dimensional subspaces of $V$.
The \emph{Grassmann graph} $J_q(N,D)$ has vertex set $X$, where two vertices are adjacent whenever their intersection has dimension $D-1$; cf. \cite[p. 268]{1989BCN}. 
We readily see that the cardinality of $X$ is $\gauss{N}{D}$.
Observe that two vertices $x,y$ have distance $i$ if and only if $\dim(x\cap y)=D-i$.
Note that $J_q(N,D)$ is isomorphic to $J_q(N,N-D)$. By our restriction on $N$ and $D$, $J_q(N,D)$ has diameter $D$.

Throughout the rest of this paper, let $\Gamma$ denote the Grassmann graph $J_q(N,D)$ with diameter $D$; to avoid trivialities we assume $D\geq 3$.
We recall some basic results that we need; cf. \cite[Section 9.3]{1989BCN}.
The graph $\Gamma$ \emph{is} distance-regular with intersection numbers given by
\begin{equation}\label{Gr.IntNum;a,b,c}
	a_i = \gauss{i}{1}\left(\gauss{i}{1}-q^{i+1}\gauss{D-i}{1}-q\gauss{N-D}{1}\right), \quad
	 b_i  = q^{2i+1}\gauss{D-i}{1}\gauss{N-D-i}{1},  \quad   
	 c_i  = \gauss{i}{1}^2, 
\end{equation}
for $0 \leq i \leq D$.
The eigenvalues of $\Gamma$ are given by

\begin{equation*}\label{Gr.Eigval}
	\theta_i = q^{i+1}\gauss{D-i}{1}\gauss{N-D-i}{1}-\gauss{i}{1}, 
\end{equation*}
for $0 \leq i \leq D$.
The graph $\Gamma$ \emph{is} $Q$-polynomial with respect to the ordering $\{\theta_i\}^D_{i=0}$ with $\theta_0 > \theta_1> \cdots> \theta_D$.
The dual eigenvalues of $\Gamma$ are given by
\begin{equation}\label{Gr.dualEigval}
	\theta^*_i = \frac{(q^N-q)(2-q^D-q^{N-D})}{(q-1)(q^D-1)(q^{N-D}-1)} + \frac{(q^N-q)(q^N-1)}{(q-1)(q^D-1)(q^{N-D}-1)}q^{-i},
\end{equation}
for $0 \leq i \leq D$; cf. \cite[Table 6.1, Theorem 8.4.1]{1989BCN}, \cite[Lemma 4.3]{2015GaGaHoLAA}.

Let $C$ be a collection of  all $D$-dimensional subspaces of $V$ containing a fixed $(D-1)$-dimensional subspace.
Then $C$ is a maximal clique\footnote{There is the other type of maximal cliques in $\Gamma$, namely, the collection of all $D$-dimensional subspaces of $V$ contained in a fixed $(D+1)$-dimensional subspace. Note that these maximal cliques are not Delsarte unless $N=2D$.} 
of $\Gamma$ and we have
\begin{equation}\label{size.C}
	|C| = \gauss{N-D+1}{1}.
\end{equation}
From this, it follows that $C$ is a \emph{Delsarte} clique, i.e., $C$ attains the Hoffman bound $1-\theta_0/\theta_D$; cf. \cite[Proposition 4.4.6]{1989BCN}.
Take a Delsarte clique $C$ of $\Gamma$.
The \emph{covering radius} of $C$ is defined by $\max\{ \partial(y,C) : y \in X \}$, where $\partial(y,C) = \min\{\partial(y,z) : z \in C\}$.
Note that the covering radius of $C$ is given by $D-1$; cf \cite[Lemma 7.4]{1993GodsilAC}.
Define
\begin{equation}\label{dist.part(C)}
	C_i := \{ y \in X :\partial(y,C) =i \}, \qquad 0 \leq i \leq D-1.
\end{equation}
For notational convenience, we set $C_{-1}:=\varnothing$ and $C_D:=\varnothing$.
We remark that $\{C_i\}^{D-1}_{i=0}$ is an \emph{equitable} partition, i.e., for all integers $i$ and $j$, $0 \leq i, j \leq D-1$, each vertex in $C_i$ has constant neighbors in $C_j$. 
In particular, for each $z\in C_i$, $0 \leq i \leq D-1$, there exist constant numbers $\widetilde a_i$, $\widetilde b_i$, $\widetilde c_i$ such that 
\begin{equation}\label{IntNum.C}
	\widetilde c_i = |\Gamma(z) \cap C_{i-1}|, \qquad 
	\widetilde a_i = |\Gamma(z) \cap C_{i}|, \qquad 
	\widetilde b_i = |\Gamma(z) \cap C_{i+1}|,
\end{equation}
where $\widetilde c_0 = \widetilde b_{D-1}=0$ and $\widetilde b_{i-1}\widetilde c_i \ne 0$ for $1 \leq i \leq D-1$. 
Observe that $\widetilde a_0 = |C|-1 = q\gauss{N-D}{1}$ and $\widetilde a_i + \widetilde b_i + \widetilde c_i = b_0$ for $0 \leq i \leq D-1$.
The constants $\widetilde a_i$, $\widetilde b_i$, $\widetilde c_i$ are called the \emph{intersection numbers} of $C$.
For $0 \leq i \leq D-1$ and $z \in C_i$, consider the subset $\{ y \in C \mid \partial(y,z)=i\}$ of $C$.
Then by the construction the cardinality of this set is given by
\begin{equation}\label{N(i)}
	n_i(z):=|\{ y \in C \mid \partial(y,z)=i\}| = \gauss{i+1}{1}, 
\end{equation}
from which it follows that the cardinality $n_i(z)$ is independent of the choice of $z$ in $C_i$, and thus we write $n_i=n_i(z)$ for $0 \leq i \leq D-1$.
By definition, we have
\begin{equation}\label{AiC=lin.combo}
	A_i\hat{C} = (|C|-n_{i-1})\hat{C}_{i-1} + n_i\hat{C}_i, \qquad 0 \leq i \leq D,
\end{equation}
where $(|C|-n_{-1})\hat{C}_{-1}=0$ and $n_D\hat{C}_D=0$.

\begin{lemma}\label{tilde;a,b,c(i)}
The intersection numbers of $C$ are given by 
\begin{align}
	& \widetilde{a}_i = \frac{1}{q-1}\left( (q^{D+1}-q^{i+1}+1)\gauss{i}{1} + q^{N-D+1}\gauss{i+1}{1} -q \gauss{2i+1}{1} \right), \label{IntNumC;a} \\
	& \widetilde{b}_i = q^{2i+2}\gauss{D-i-1}{1}\gauss{N-D-i}{1}, \qquad  \qquad 
	\widetilde{c}_i = \gauss{i+1}{1}\gauss{i}{1}, \label{IntNumC;bc}
\end{align} 
for $0 \leq i \leq D-1$.
\end{lemma}
\begin{proof}
We recall the intersection numbers $b_i$, $c_i$ of $\Gamma$.
By \cite[Theorem 4.7]{2013LeeLAA} and \eqref{N(i)},
\begin{equation}\label{rels:tbc-bc}
	\widetilde b_i = \frac{q^{D-N+i}-1}{q^{D-N+i+1}-1}b_{i+1}, \qquad  \widetilde c_i = \frac{q^{i+1}-1}{q^i-1}c_i, \qquad 0\leq i \leq D-1.
\end{equation}
Evaluate \eqref{rels:tbc-bc} using \eqref{Gr.IntNum;a,b,c} to get \eqref{IntNumC;bc}.
To verify \eqref{IntNumC;a}, use $\widetilde a_i + \widetilde b_i + \widetilde c_i = b_0$ along with \eqref{Gr.IntNum;a,b,c}. 
\end{proof}

We recall the Terwilliger algebra associated with $C$ in the sense of Suzuki \cite{2005SuzukiJACO}.
For $0\leq i \leq D-1$, let $\widetilde E^*_i=\widetilde E^*_i(C)$ denote the diagonal matrix in $\mathrm{Mat}_X(\mathbb{C})$ with $(y,y)$-entry $1$ if $y \in C_i$ and $0$ otherwise, i.e., $\widetilde E^*_i = \mathrm{diag}(\hat{C}_i)$.
Observe that $\widetilde E^*_i \widetilde E^*_j = \delta_{ij} \widetilde E^*_i$, $0 \leq i,j \leq D-1$ and $\sum^{D-1}_{i=0}\widetilde E^*_i = I$.
The \emph{dual Bose-Mesner algebra} of $\Gamma$ with respect to $C$ is the  (commutative) semisimple subalgebra $\widetilde M^*=\widetilde M^*(C)$ of $\mathrm{Mat}_X(\mathbb{C})$ generated by $\widetilde E^*_0, \widetilde E^*_1, \ldots, \widetilde E^*_{D-1}$.
Note that the matrices $\widetilde E^*_i$, $0 \leq i \leq D-1$, form a basis for $\widetilde M^*$.
Define the diagonal matrix $\widetilde A^*=\widetilde A^*(C)$ in $\mathrm{Mat}_X(\mathbb{C})$ by 
\begin{equation*}\label{dualadjmat(C)}
	\widetilde A^* = \frac{1}{|C|}\sum_{y \in C} A^*(y) = \frac{|X|}{|C|}\mathrm{diag}(E_1\hat{C}).
\end{equation*}
Since $E_1=|X|^{-1}\sum^D_{i=0}\theta^*_iA_i$ and by \eqref{AiC=lin.combo}, we have
\begin{equation}\label{E1C=sum(tilE_i)}
	\frac{|X|}{|C|}\mathrm{diag}(E_1\hat{C}) = \sum^{D-1}_{i=0} \left(\frac{n_i\theta^*_i+(|C|-n_{i})\theta^*_{i+1}}{|C|}\right)\widetilde E^*_i.
\end{equation}
For $0 \leq i \leq D-1$, let $\widetilde \theta^*_i$ denote the coefficient of each summand of $\widetilde E^*_i$ in \eqref{E1C=sum(tilE_i)}.
By \eqref{Gr.dualEigval}, \eqref{size.C} and \eqref{N(i)}, the $\widetilde\theta^*_i$ are given by
\begin{equation}\label{Gr.dualEigval.C}
	\widetilde \theta^*_i =  \frac{(q^{N-1}-1)(q+q^2-q^{D+1}-q^{N-D+2})}{(q-1)(q^D-1)(q^{N-D+1}-1)} + \frac{(q^N-q)(q^N-1)}{(q-1)(q^D-1)(q^{N-D+1}-1)}q^{-i},
\end{equation}
for $0 \leq i \leq D-1$.
Observe that the scalars $\widetilde\theta^*_i$ are real and mutually distinct.
We write
\begin{equation}\label{tilde(A);linComb(tildeE*i)}
	\widetilde A^* = \sum^{D-1}_{i=0}\widetilde\theta^*_i \widetilde E^*_i,
\end{equation}
from which it follows that $\widetilde A^*$ generates $\widetilde M^*$.
We call $\widetilde A^*$ the \emph{dual adjacency matrix of $\Gamma$ with respect to $C$}. 
Observe that the scalars $\widetilde\theta^*_i$ are the eigenvalues of $\widetilde A^*$, called the \emph{dual eigenvalues} of $\Gamma$ with respect to $C$.
The Terwilliger algebra of $\Gamma$ with respect to $C$ is the subalgebra $\widetilde T=\widetilde T(C)$ of $\mathrm{Mat}_X(\mathbb{C})$ generated by $M, \widetilde M^*$; cf. \cite{2005SuzukiJACO}.
Note that the matrices $A$, $\widetilde A^*$ generate $\widetilde T$.
Note also that $\widetilde T$ is (non-commutative) semisimple.
The following are relations in $\widetilde T$:
\begin{equation*}
	\widetilde E^*_iA\widetilde E^*_j = 0, \qquad  E_i\widetilde A^*E_j=0 \qquad \text{if} \quad |i-j|>1,
\end{equation*}
for $0 \leq i, j \leq D$, where we set $\widetilde E^*_D=0$; cf. \cite[Section 4]{2005SuzukiJACO}.

We note that the subspace $M\hat{C}$ of $\mathbb{C}^X$ has bases $\{A_i\hat{C}\}^{D-1}_{i=0}$, $\{\hat{C}_i\}^{D-1}_{i=0}$, and $\{E_i\hat{C}\}^{D-1}_{i=0}$.
By \eqref{AiC=lin.combo} and $\hat{C}_i = \widetilde E^*_i\hat{X}$, $0\leq i \leq D-1$,
the subspace $M\hat{C}$ is same as the subspace $\widetilde M^*\hat{X}$ of $\mathbb{C}^X$, and therefore $M\hat{C}$ is an irreducible $\widetilde T$-module, called the \emph{primary} $\widetilde T$-module.
The actions of $A, \widetilde A^*$ on $M\hat{C}$ are given as follows: for $0 \leq i \leq D-1$,
\begin{equation*}
	A.\hat{C}_i = \widetilde b_{i-1}\hat{C}_{i-1} + \widetilde a_i\hat{C}_i+\widetilde c_{i+1}\hat{C}_{i+1}, \qquad   \widetilde A^*.\hat{C}_i = \widetilde\theta^*_i\hat{C}_i.
\end{equation*}

\section{The generalized Terwilliger algebra of Grassmann graphs}\label{Section:algebraT}

We continue to discuss the Grassmann graph $\Gamma=J_q(N,D)$.
Throughout the rest of the paper, we fix a Delsarte clique $C$ of $\Gamma$ and a vertex $x$ in $C$.
We recall the Terwilliger algebras $T=T(x)$ and $\widetilde T=\widetilde T(C)$ of $\Gamma$.
In this section, we treat the generalized Terwilliger algebra of $\Gamma$ associated with $x$ and $C$, and discuss its so-called primary module.

\begin{definition}[{\cite[Definition 5.20]{2013LeeLAA}}]\label{Algebra.T}
The \emph{generalized Terwilliger algebra} of $\Gamma$ with respect to $x$, $C$  is  the subalgebra $\mathbf{T}=\mathbf{T}(x,C)$ of $\mathrm{Mat}_X(\mathbb{C})$ generated by $T$, $\widetilde T$.
Note that $A$, $A^*$, $\widetilde A^*$ generate $\mathbf{T}$, where $A^*\widetilde A^* = \widetilde A^* A^*$, and that $\mathbf{T}$ is (non-commutative) semisimple.
\end{definition}

Recall two partitions $\{\Gamma_i(x)\}^D_{i=0}$ and $\{C_i\}^{D-1}_{i=0}$ of $X$ from \eqref{dist.part(x)} and \eqref{dist.part(C)}, respectively.
Using these, we define a new partition $\{C^\pm_i\}^{D-1}_{i=0}$ of $X$ by
\begin{equation}\label{new.partition}
	C^-_i := C_i \cap \Gamma_{i}(x), \qquad  
	C^+_i := C_i \cap \Gamma_{i+1}(x), \qquad 0 \leq i \leq D-1.
\end{equation}
See Figure \ref{2-dim'l partition}.
For notational convenience, we set $C^-_{-1}=C^+_{-1}=\varnothing$ and $C^-_{D}=C^+_{D}=\varnothing$.
Observe that $C_i = C^-_i \cup C^+_i$, $0 \leq i \leq D-1$, and $\Gamma_i(x) = C^+_{i-1} \cup C^-_i$, $0 \leq i \leq D$.
In particular, $x=C^-_0$ and $C=C^-_0 \cup C^+_0$.
From this and \eqref{size.C}, it easily follows that $|C^+_0| = \frac{q(q^{N-D}-1)}{q-1}$.
\begin{figure}
\centering
\scalemath{0.65}{
\begin{tikzpicture}
  [scale=1,thick,auto=left,every node/.style={circle,draw}] 
  \node (n1) at (0,0) {$C^-_0$};
  \node (n2) at (0,-2)  {$C^+_0$};
  \node (n3) at (2,-2)  {$C^-_1$};
  \node (n4) at (2,-4) {$C^+_1$};
  \node (n5) at (4,-4)  {$C^-_2$};
  \node (n6) at (4,-6)  {$C^+_2$};   
  \node (n7) at (6,-6)  {$C^-_3$};
  \node (n8) at (6,-8)  {$C^+_3$};
 
  \node[fill=black!10] (c0) at (0,2.5) {${~}C^{~}_0$};
  \node[fill=black!10] (c1) at (2,2.5) {${~}C^{~}_1$};
  \node[fill=black!10] (c2) at (4,2.5) {${~}C^{~}_2$};
  \node[fill=black!10] (c3) at (6,2.5) {${~}C^{~}_3$};
  
  \node[fill=blue!20] (r0) at (-2.5,0) {${~}\Gamma_0$};
  \node[fill=blue!20] (r1) at (-2.5,-2) {${~}\Gamma_1$};
  \node[fill=blue!20] (r2) at (-2.5,-4) {${~}\Gamma_2$};
  \node[fill=blue!20] (r3) at (-2.5,-6) {${~}\Gamma_3$};
  \node[fill=blue!20] (r4) at (-2.5,-8) {${~}\Gamma_4$};

  \foreach \from/\to in {n1/n2,n2/n3,n3/n4,n4/n5,n5/n6,n6/n7,n7/n8,n1/n3,n3/n5,n5/n7,n2/n4,n4/n6,n6/n8, c0/c1,c1/c2,c2/c3, r0/r1, r1/r2, r2/r3, r3/r4}
    \draw (\from) -- (\to)[line width=0.7mm] ;
    
    \draw (c0) -- (n1) [dashed];
    \draw (c1) -- (n3) [dashed];
    \draw (c2) -- (n5) [dashed];
    \draw (c3) -- (n7) [dashed];
        
    \draw (r0) -- (n1) [dashed];
    \draw (r1) -- (n2) [dashed];
    \draw (r2) -- (n4) [dashed];
    \draw (r3) -- (n6) [dashed];
    \draw (r4) -- (n8) [dashed];

\end{tikzpicture}}
\caption{The partition $\{C^{\pm}_i\}^{D-1}_{i=0}$ of $X$ when $D=4$}\label{2-dim'l partition}
\end{figure}
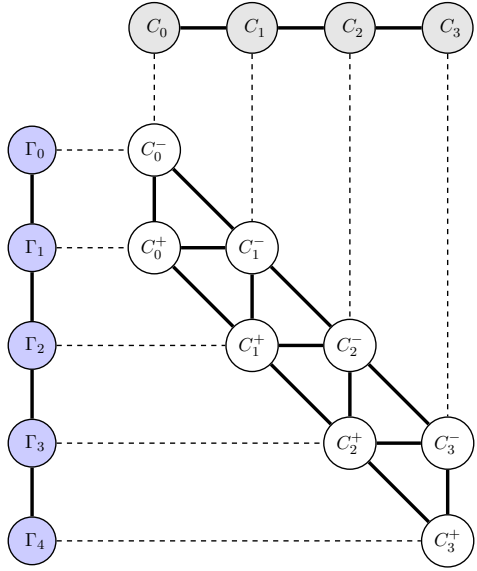
\begin{lemma}\label{cardinality:C+-}
For $0 \leq i \leq D-1$, the cardinality of each cell $C^{\pm}_i$ is given by
\begin{align*}
	|C^-_i| & = q^{i(i+1)}\prod^i_{j=1}  \frac{(q^{D-j}-1)(q^{N-D+1-j}-1)}{(q^j-1)^2}, \\
	|C^+_i| & = \frac{q^{(i+1)^2}(q^{N-D}-1)}{q-1}\prod^i_{j=1}  \frac{(q^{D-j}-1)(q^{N-D-j}-1)}{(q^j-1)(q^{j+1}-1)}.
\end{align*}
In particular, each of $C^\pm_i$ is non-empty.
\end{lemma}
\begin{proof}
Since $\Gamma$ is distance-regular and the partition $\{C_i\}^{D-1}_{i=0}$ is equitable, by \eqref{IntNum.Gamma}, \eqref{IntNum.C}, and \eqref{new.partition}, it follows
\begin{equation}\label{pf:|C-|,|C+|;eq}
	\widetilde b_i |C^-_i| = c_{i+1}|C^-_{i+1}|, \qquad b_{i+1}|C^+_i| = \widetilde c_{i+1}|C^+_{i+1}|, \qquad 0 \leq i \leq D-2.
\end{equation}
Evaluate \eqref{pf:|C-|,|C+|;eq} using \eqref{Gr.IntNum;a,b,c}, \eqref{IntNumC;bc}  and use induction on $i$ with $|C^-_0|=1$ and $|C^+_0| = \frac{q(q^{N-D}-1)}{q-1}$.
\end{proof}

We remark that from \eqref{new.partition} it turns out that the partition $\{C^\pm_{i}\}^{D-1}_{i=0}$ \emph{is} equitable; cf. \cite[Lemmas 5.1, 5.2]{2013LeeLAA}.
Let $\mathbf{W}$ be a subspace of $\mathbb{C}^X$ spanned by the set
\begin{equation}\label{Basis.W}
	\mathcal{C}:=\{\hat{C}^-_0, \hat{C}^+_0,\hat{C}^-_1, \hat{C}^+_1, \ldots, \hat{C}^-_{D-1}, \hat{C}^+_{D-1}\}.
\end{equation}
Observe that $\mathcal{C}$ is an orthogonal ordered basis for $\mathbf{W}$.
Since $\{C^\pm_i\}^{D-1}_{i=0}$ is equitable, $\mathbf{W}$ is $A$-invariant. 
Moreover, by the construction of \eqref{new.partition}, $\mathbf{W}$ is a module for both $M^*$ and $\widetilde M^*$.
Therefore, $\mathbf{W}$ is a $\mathbf{T}$-module.
Note that the $\mathbf{T}$-module $\mathbf{W}$ is generated by $\hat{x}$ since 
\begin{equation*}
	E^*_i\widetilde E^*_i\mathbf{J}\hat{x} = \hat{C}^-_i, \qquad 
	E^*_{i+1}\widetilde E^*_i\mathbf{J}\hat{x} = \hat{C}^+_i, \qquad
	0 \leq i \leq D-1,
\end{equation*}
where $\mathbf{J}=\sum^D_{i=0}A_i$ and observe that $\mathbf{J}\hat{x}=\hat{X}$.
\begin{lemma}[{cf. \cite[Proposition 5.25]{2013LeeLAA}}]
The $\mathbf{T}$-module $\mathbf{W}$ is irreducible.
\end{lemma}
\begin{proof}
By semisimplicity of $\mathbf{T}$, $\mathbf{W}$ decomposes into an orthogonal direct sum of irreducible $\mathbf{T}$-modules.
Among such modules, take one, denoted by $\mathbf{W}_0$, which is not orthogonal to $\hat{x}$.
Then $E^*_0\mathbf{W}_0$ contains $\hat{x}$, from which it follows that the irreducible $\mathbf{T}$-module $\mathbf{W}_0$ contains $\hat{x}$.
Since the $\mathbf{T}$-module $\mathbf{W}$ is generated by $\hat{x}$, we have $\mathbf{W}=\mathbf{W}_0$. The result follows.
\end{proof}
We remark that the irreducible $\mathbf{T}$-module $\mathbf{W}$ is generated by $\hat{C}$ as well. We call $\mathbf{W}$ the \emph{primary} $\mathbf{T}$-module.
We describe the action of $\mathbf{T}$ on the basis $\mathcal{C}$ for $\mathbf{W}$.
Note that $\hat{C}^\pm_{-1}=0$ and $\hat{C}^\pm_{D}=0$.
\begin{lemma}\label{lem:actionAdjMatA}
The action of $A$ on $\hat{C}^\pm_i$, $0 \leq i \leq D-1$, is given by	
\begin{align*}
	A.\hat{C}^-_i 
	& = q^{2i}\gauss{D-i}{1}\gauss{N-D+1-i}{1}\hat{C}^-_{i-1} + q^{2i}\gauss{D-i}{1}\hat{C}^+_{i-1} \\
	& \qquad +  \left(q\gauss{D}{1}\gauss{N-D}{1} - q^{2i+1}\gauss{D-i}{1}\gauss{N-D-i}{1} - \gauss{i+1}{1}\gauss{i}{1} \right) \hat{C}^-_{i} \\
	& \qquad + q^i\gauss{i+1}{1} \hat{C}^+_{i} + \gauss{i+1}{1}^2 \hat{C}^-_{i+1}, \\
	A.\hat{C}^+_i
	& = 	q^{2i+1}\gauss{D-i}{1}\gauss{N-D-i}{1}\hat{C}^+_{i-1} + q^{2i+1}\gauss{N-D-i}{1} \hat{C}^-_{i} \\
	& \qquad + \left(q\gauss{D}{1}\gauss{N-D}{1} - q^{2i+2}\gauss{D-1-i}{1}\gauss{N-D-i}{1} - \gauss{i+1}{1}^2 \right)\hat{C}^+_{i} \\
	& \qquad +  q^{i+1}\gauss{i+1}{1} \hat{C}^-_{i+1} + \gauss{i+2}{1}\gauss{i+1}{1} \hat{C}^+_{i+1}.
	\end{align*}
\end{lemma}
\begin{proof}
From the structure of \eqref{new.partition}, we routinely find both
\begin{align*}
	A.\hat{C}^-_i & = \widetilde b_{i-1}\hat{C}^-_{i-1} + (\widetilde b_{i-1}-b_i)\hat{C}^+_{i-1} + (\widetilde a_i-b_i+\widetilde b_i) \hat{C}^-_i + (c_{i+1}-\widetilde c_i)\hat{C}^+_i + c_{i+1}\hat{C}^-_{i+1},\\ 
	A.\hat{C}^+_i & = b_{i}\hat{C}^+_{i-1} + (b_{i}-\widetilde b_i)\hat{C}^-_{i} + (\widetilde a_i-c_{i+1}+\widetilde c_i) \hat{C}^+_i + (\widetilde c_{i+1}- c_{i+1})\hat{C}^-_{i+1} + \widetilde c_{i+1}\hat{C}^+_{i+1}, 
\end{align*}
for $0 \leq i \leq D-1$.
Evaluate these equations using \eqref{Gr.IntNum;a,b,c} and Lemma \ref{tilde;a,b,c(i)}. The result follows.
\end{proof}

\begin{lemma}\label{lem:actionMatA*,wtA*}
The actions of $A^*$, $\widetilde A^*$ on $\hat{C}^\pm_i$, $0 \leq i \leq D-1$, are given by
\begin{align*}
	&&& A^*.\hat{C}^-_i = \theta^*_i \hat{C}^-_i, && A^*.\hat{C}^+_i = \theta^*_{i+1}\hat{C}^+_i, &&\\
	&&& \widetilde A^*.\hat{C}^-_i = \widetilde \theta^*_i \hat{C}^-_i, && \widetilde A^*.\hat{C}^+_i = \widetilde \theta^*_{i}\hat{C}^+_i, &&
\end{align*}
where $\theta^*_i$ are from \eqref{Gr.dualEigval} and $\widetilde\theta^*_i$ are from \eqref{Gr.dualEigval.C}.
\end{lemma}
\begin{proof}
Immediate from \eqref{A*;linComb(E*i)} and \eqref{tilde(A);linComb(tildeE*i)}.
\end{proof}

\section{Leonard systems of dual $q$-Hahn type}\label{Section:LS of dual q-Hahn} 

In this section, we discuss a family of Leonard systems said to have dual $q$-Hahn type and some properties we need in the paper.
We begin by recalling the notion of Leonard systems \cite{2001TerLAA}.
Let $\sfd$ be a non-negative integer and let $\mathsf V$ be a $\mathbb{C}$-vector space with dimension $\sfd+1$.
Assume that the element $\mathsf A \in \mathrm{End}(\mathsf V)$ is \emph{multiplicity-free}, i.e., $\mathsf A$ has $\sfd+1$ mutually distinct eigenvalues $\mytheta_0, \mytheta_1, \ldots, \mytheta_\sfd$.
For $0\leq i \leq \sfd$, define $\mathsf E_i \in \mathrm{End}(\mathsf V)$ such that
$$
\mathsf E_i=\prod_{\substack{ 0 \le j \le \mathsf{d} \\ j\ne i} } \frac{\mathsf A-\mytheta_j\mathsf I}{\mytheta_i-\mytheta_j},
$$
where $\mathsf I$ is the identity of $\mathrm{End}(\mathsf V)$.
Observe that (i) $\mathsf A\mathsf E_i = \mytheta_i\mathsf E_i$, $0 \leq i \leq \sfd$,  (ii) $\mathsf E_i\mathsf E_j = \delta_{ij}\mathsf E_i$, $0 \leq i,j \leq \sfd$, and (iii) $\sum^{\sfd}_{i=0} \mathsf E_i = \mathsf I$.
We call $\mathsf E_i$ the \emph{primitive idempotent} of $\mathsf A$ associated with $\mytheta_i$.

\begin{definition}[{\cite[Definition 1.4]{2001TerLAA}}]\label{Def:LS}
By a \emph{Leonard system} on $\mathsf V$, we mean a sequence
\begin{equation}\label{def:LS}
	\mathsf \Phi = \{ \mathsf A, \mathsf A^*, \{\mathsf E_i\}^{\sfd}_{i=0}, \{\mathsf E^*_i\}^{\sfd}_{i=0} \}
\end{equation}
of elements in $\mathrm{End}(\mathsf{V})$ that satisfy (i)--(iii) below.
\begin{itemize}
	\item[(i)] Each of $\mathsf{A}, \mathsf{A}^*$ is multiplicity-free in $\mathrm{End}(\mathsf{V})$.
	\item[(ii)] $\{\mathsf E_i\}^{\sfd}_{i=0}$ (resp. $\{\mathsf E^*_i\}^{\sfd}_{i=0}$) is an ordering of the primitive idempotents of $\mathsf A$ (resp. $\mathsf A^*$).
	\item[(iii)] For $0 \leq i, j \leq \sfd$, both
	\begin{equation}
		\mathsf E_i \mathsf A^* \mathsf E_j = \begin{cases} 0 & \textrm{ if } |i-j|>1, \\ \ne 0 & \textrm{ if } |i-j|=1,  \end{cases}
		\quad \text{ and } \quad 
		\mathsf E^*_i \mathsf A \mathsf E^*_j = \begin{cases} 0 & \textrm{ if } |i-j|>1, \\ \ne 0 & \textrm{ if } |i-j|=1.  \end{cases} \qquad 
	\end{equation}
\end{itemize}
We call $\sfd$ the \emph{diameter} of $\mathsf \Phi$. 
\end{definition}
\begin{note}
In a common notational convention, $\mathsf A^*$ denotes the conjugate-transpose of $\mathsf A$.
We are not using this convention. 
The elements $\mathsf A, \mathsf A^*$ in $\eqref{def:LS}$ are arbitrary subject to (i)--(iii) above.
\end{note}
Let $\mathsf \Phi=\{ \mathsf A, \mathsf A^*, \{\mathsf E_i\}^{\sfd}_{i=0}, \{\mathsf E^*_i\}^{\sfd}_{i=0} \}$ be a Leonard system on $\mathsf V$.
Let $\mathsf \Phi'$ be a Leonard system on a $(\sfd+1)$-dimensional  $\mathbb{C}$-vector space $\mathsf V'$.
We say that $\mathsf \Phi'$ is \emph{isomorphic} to $\mathsf \Phi$ if there is a $\mathbb{C}$-algebra isomorphism $\sigma: \mathrm{End}(\mathsf{V}) \to \mathrm{End}(\mathsf{V'})$ such that $\mathsf \Phi'=\mathsf \Phi^\sigma = \{ \mathsf A^\sigma, \mathsf A^{*\sigma}, \{\mathsf E^\sigma_{i}\}^{\sfd}_{i=0}, \{\mathsf E^{*\sigma}_i\}^{\sfd}_{i=0} \}$.
Consider two sequences
\begin{equation}
	\mathsf \Phi^*=\{ \mathsf A^*, \mathsf A, \{\mathsf E^*_{i}\}^{\sfd}_{i=0}, \{\mathsf E_i\}^{\sfd}_{i=0} \}, \qquad
	\mathsf \Phi^{\Downarrow}=\{ \mathsf A, \mathsf A^*, \{\mathsf E_{\sfd-i}\}^{\sfd}_{i=0}, \{\mathsf E^*_i\}^{\sfd}_{i=0} \}
\end{equation}
Then both $\mathsf \Phi^*$ and $\mathsf \Phi^{\Downarrow}$ satisfy the conditions (i)--(iii) in Definition \ref{Def:LS}, and thus they are Leonard systems on $\mathsf V$.
For $0 \leq i \leq \sfd$, let $\mytheta_i$ (resp. $\mytheta^*_i$) be an eigenvalue of $\mathsf A$ (resp. $\mathsf A^*$).
Then there exists nonzero scalars $\myvarphi_i$, $1 \leq i \leq \sfd$, and an isomorphism of $\mathbb{C}$-algebras $\natural$ from $\mathrm{End}(\mathsf{V})$ to the full matrix algebra $\mathbb{C}^{(\sfd+1)\times(\sfd+1)}$ such that (cf. \cite[Theorem 3.2]{2001TerLAA})
\begin{equation}
	A^\natural = \begin{bmatrix}
	\mytheta_0 & & &  & {\mathbf 0} \\[0.2em] 
	1 & \mytheta_1 & \\[0.2em] 
	& 1 & \mytheta_2 &  \\
	& & \ddots & \ddots \\[0.2em] 
	{\mathbf 0}& & & 1 & \mytheta_\sfd
	\end{bmatrix},
	\qquad
	A^{*\natural} = \begin{bmatrix}
	\mytheta^*_0 & \myvarphi_1 & &  & {\mathbf 0} \\[0.2em] 
	& \mytheta^*_1 & \myvarphi_2\\
	& & \mytheta^*_2 & \ddots \\
	& & & \ddots & \myvarphi_\sfd\\[0.2em] 
	{\mathbf 0}& & & & \mytheta^*_\sfd
	\end{bmatrix}.
\end{equation}
We call the sequence $\{\myvarphi_i\}^{\sfd}_{i=1}$ the \emph{first split sequence} of $\mathsf \Phi$.
Let $\{\myphi_i\}^{\sfd}_{i=1}$ denote the first split sequence of $\mathsf \Phi^{\Downarrow}$ and call this the \emph{second split sequence} of $\mathsf \Phi$. 
By the \emph{parameter array} of $\mathsf \Phi$, we mean the sequence
\begin{equation}\label{Def:PA}
	( \{\mytheta_i\}^{\sfd}_{i=0}, \{\mytheta^*_i\}^{\sfd}_{i=0}, \{\myvarphi_i\}^{\sfd}_{i=1}, \{\myphi_i\}^{\sfd}_{i=1} ).
\end{equation}
Take a non-zero vector $u$ in $\mathsf E_0 \mathsf V$.
Then the set $\{\mathsf E^*_iu\}^{\sfd}_{i=0}$ forms a \emph{$\mathsf \Phi$-standard basis}\footnote{Dually, we can consider a $\mathsf \Phi^*$-standard basis  $\{\mathsf E_iu^*\}^{\sfd}_{i=0}$ for $\mathsf V$ with a non-zero $u^* \in \mathsf E^*_0\mathsf V$.} for $\mathsf V$, i.e., the set $\{\mathsf E^*_iu\}^{\sfd}_{i=0}$ satisfies both (i) $\mathsf E^*_iu \in \mathsf E^*_i\mathsf V$, $0\leq i \leq \sfd$; (ii) $\sum^\sfd_{i=0}E^*_iu \in \mathsf E_0\mathsf V$.
Applying $\mathsf A$ to $\mathsf E^*_iu$ and using Definition \ref{Def:LS}(iii), there exist the scalars $\sfa_i$, $\sfb_i$, $\sfc_i$, $0 \leq i \leq \sfd$, the so-called \emph{intersection numbers} of $\mathsf\Phi$, such that
$\sfb_\sfd=\sfc_0=0$, $\sfb_{i-1}\sfc_i \ne 0$, $1\leq i \leq \sfd$, and
\begin{equation}\label{eq;3-term.AE*i}
	\mathsf A \mathsf E^*_i u = \sfb_{i-1}\mathsf E^*_{i-1}u + \sfa_i\mathsf E^*_i u + \sfc_{i+1}\mathsf E^*_{i+1}u,
\end{equation}
where $\sfb_{-1}\mathsf E^*_{-1}u=0$ and $\sfc_{\sfd+1}\mathsf E^*_{\sfd+1}u=0$.
Note that $\sfa_i+\sfb_i+\sfc_i = \mytheta_0$ for $0 \leq i \leq \sfd$.
The intersection numbers $\sfb_i$ and $\sfc_i$ are given in terms of the parameter array \eqref{Def:PA} by (cf. \cite[Theorem 17.7]{2004TerLAA})
\begin{align}
	\sfb_i &= \myvarphi_{i+1}\frac{(\mytheta^*_i-\mytheta^*_0)(\mytheta^*_i-\mytheta^*_1)\cdots(\mytheta^*_i-\mytheta^*_{i-1})}{(\mytheta^*_{i+1}-\mytheta^*_0)(\mytheta^*_{i+1}-\mytheta^*_1)\cdots (\mytheta^*_{i+1}-\mytheta^*_i)}, \qquad 0 \leq i \leq \sfd-1,  \label{i.n.Phi;b} \\
	\sfc_i &= \myphi_{i}\frac{(\mytheta^*_i-\mytheta^*_{i+1})(\mytheta^*_i-\mytheta^*_{i+2})\cdots(\mytheta^*_i-\mytheta^*_{\sfd})}{(\mytheta^*_{i-1}-\mytheta^*_i)(\mytheta^*_{i-1}-\mytheta^*_{i+1})\cdots (\mytheta^*_{i-1}-\mytheta^*_\sfd)}, \qquad 1 \leq i \leq \sfd. \label{i.n.Phi;c}
\end{align}
Using the intersection numbers $\sfa_i$, $\sfb_i$, $\sfc_i$, define a sequence of polynomials $\{\mathsf v_i\}^{\sfd}_{i=0}$ in $\mathbb{C}[\lambda]$ as follows: 
\begin{equation}\label{eq:poly.v(i)}
	\mathsf v_0:=1, \qquad
	\lambda \mathsf v_i = \sfb_{i-1}\mathsf v_{i-1} + \sfa_i\mathsf v_i + \sfc_{i+1}\mathsf v_{i+1}, \quad 0 \leq i \leq \sfd-1,
\end{equation}
where $\sfb_{-1}\mathsf v_{-1}=0$.
Observe that $\deg(\mathsf v_i)=i$  for  $0 \leq i \leq \sfd$ since $\mathsf c_j \ne 0$, $1 \leq j \leq \sfd$.
We say that the polynomial $\mathsf v_i$ is \emph{associated with} $\mathsf\Phi$.
By \eqref{eq;3-term.AE*i}, it follows
\begin{equation}\label{vi(A)E*i}
	\mathsf v_i(\mathsf A). \mathsf E^*_0 u=\mathsf E^*_iu, \qquad 0 \leq i \leq \sfd.
\end{equation}
We normalize the polynomial $\mathsf v_i$ by setting 
\begin{equation}\label{poly:f}
	\mathsf f_i:=\mathsf v_i/\mathsf k_i, \qquad 0 \leq i \leq \sfd,
\end{equation}
where $\mathsf k_i = \sfb_0\sfb_1\cdots \sfb_{i-1}/\sfc_1\sfc_2\cdots\sfc_i$.
Then it turns out that (cf. \cite[Theorem 17.4]{2004TerLAA})
\begin{equation}\label{poly;f}
	\mathsf f_i(\lambda) = \sum^i_{n=0} \frac{(\mytheta^*_i-\mytheta^*_0)(\mytheta^*_i-\mytheta^*_1)\cdots (\mytheta^*_i-\mytheta^*_{n-1})(\lambda-\theta_0)\cdots (\lambda-\theta_{n-1})}{\myvarphi_1\myvarphi_2\cdots \myvarphi_n}, \qquad 0 \leq i \leq \sfd.
\end{equation}

The Leonard system is uniquely determined up to isomorphism by the parameter array, cf. \cite[Theorem 1.9]{2001TerLAA}, and all families of the parameter arrays of Leonard systems are displayed in \cite{2005TerDCC} as parametric form.
We now recall the dual $q$-Hahn family of Leonard systems.
For the rest of this section, assume  that $q$ is a nonzero scalar such that $q^i \ne 1$ for $1 \leq i \leq D$.

\begin{definition}[{\cite[Example 5.5]{2005TerDCC}}]\label{Def:LS.dualqHahn}
Let $\mathsf \Phi$ be a Leonard system on $\mathsf V$ with diameter $\sfd$.
Let the sequence \eqref{Def:PA} be the parameter array of $\mathsf\Phi$.
Then $\mathsf \Phi$ is said to have \emph{dual $q$-Hahn} type if there exist scalars $\sfa$, $\sfa^*$, $\sfb$, $\sfb^*$, $\sfc$, $\sfr$ such that 
\begin{align*}
	\mytheta_i  = \sfa+\sfb q^{-i} + \sfc q^i, \qquad \qquad
	\mytheta^*_i  = \sfa^* + \sfb^*q^{-i},
\end{align*}
for $0 \leq i \leq \sfd$, and
\begin{align*}
	\myvarphi_i & = \sfb\sfb^*q^{1-2i}(1-q^i)(1-q^{i-\sfd-1})(1-\sfr q^i), \\
	\myphi_i & = \sfc\sfb^* q^{\sfd+1-2i}(1-q^i)(1-q^{i-\sfd-1})(1-\sfb\sfr\sfc^{-1}q^{i-\sfd}),
\end{align*}
for $1 \leq i \leq \sfd$, where $\mathsf b$, $\mathsf b^*$,  $\mathsf c$, $\mathsf r$ are nonzero\footnote{In the case $\mathsf r=0$, the Leonard system $\mathsf \Phi$ has dual $q$-Krawtchouk type; cf. \cite[Definition 5.2]{2018LeeTanakaSIGMA}} and neither of $\mathsf rq^i$, $\mathsf c \mathsf b^{-1} \mathsf r^{-1} q^{i-1}$ is equal to $1$ for $1\leq i \leq \mathsf d$.
We call $(\sfa, \sfa^*, \sfb, \sfb^*,\sfc, \sfr; q, \sfd)$ the \emph{parameter sequence of $\mathsf\Phi$}.
\end{definition}
From now on, let $\mathsf \Phi$ be a Leonard system of dual $q$-Hahn type as in Definition \ref{Def:LS.dualqHahn}.
From \eqref{i.n.Phi;b}, \eqref{i.n.Phi;c}, the intersection numbers of $\mathsf \Phi$ are given by
\begin{equation}\label{IntNum(dualqHahn)}
	\sfb_i  = \sfb(1-q^{i-\sfd})(1-\sfr q^{i+1}), 
	\qquad
	\sfc_i  = (1-q^{i})(\sfc-\sfb\sfr q^{i-\sfd}), 
\end{equation}
for $0 \leq i \leq \sfd$.
Evaluate \eqref{poly;f} at $\lambda=\mytheta_j$ using Definition \ref{Def:LS.dualqHahn}. 
Then we get (cf. \cite[Example 5.5]{2005TerDCC})
\begin{equation}
	\mathsf f_i(\mytheta_j) 
	= {}_3\phi_2 \left( \left. \begin{array}{c} q^{-i}, \ q^{-j}, \ \mathsf t^2 q^j \\ \sfr q, \ q^{-\sfd} \end{array}\right| q, q\right), \qquad 0 \leq i, j \leq \sfd,
\end{equation}
where
\begin{equation}\label{scalar;t.b.c}
	\mathsf t^2 = \sfb^{-1}\sfc.
\end{equation}
The polynomials $\mathsf f_i$ form the \emph{dual $q$-Hahn polynomials} \cite[Section 14.7]{2010KLS} in a variable $\lambda(x) = \sfa+\sfb q^{-x} + \sfc q^x$.
For notational convenience, fix a square root $\mathsf t$ of $\mathsf t^2$.
Set $x = \log_q(\mathsf t^{-1}\zeta)$ in $\lambda(x)$ so that
\begin{equation}
	\lambda=\lambda(\log_q(\mathsf t^{-1}\zeta))=\sfa + \sfb\mathsf{t}\zeta^{-1} + \sfc\mathsf{t}^{-1}\zeta.
\end{equation}
We renormalize $\mathsf f_i(\lambda)$ by setting
\begin{equation}\label{eq;poly.h(i)}
	\mathsf h_i(\zeta) 
	:= \frac{(\sfr q;q)_i(q^{-\sfd};q)_i}{\mathsf{t}^i} \mathsf f_i(\lambda) 
	= \frac{(\sfr q;q)_i(q^{-\sfd};q)_i}{\mathsf{t}^i} {}_3\phi_2 \left( \left. \begin{array}{c} q^{-i}, \ \mathsf{t}\zeta^{-1}, \ \mathsf{t}\zeta \\ \sfr q, \ q^{-\sfd} \end{array}\right| q, q\right),
\end{equation}
for  $0 \leq i \leq \sfd$.
We note that $\mathsf h_i(\zeta)$ are \emph{monic} symmetric Laurent polynomials in a variable $\zeta$, i.e., the coefficient of their highest degree term in $\zeta$ is one, and note also that the $\mathsf h_i(\zeta)$ has the highest degree  $i$ and the lowest degree $-i$.
Since $\mathsf h_i(\zeta)$ depends on the parameters $\mathsf b$, $\mathsf c$, $\mathsf r$, $\mathsf d$, and $q$, we write 
\begin{equation}\label{LaurentPoly;h(i)}
	\mathsf h_i = \mathsf h_i(\zeta) := \mathsf h_i(\zeta; \mathsf b, \mathsf c,\mathsf r, \sfd; q), \qquad 0 \leq i \leq \sfd,
\end{equation}
and say that $\mathsf{h}_i$ is \emph{associated with} $\mathsf\Phi$.

\begin{lemma}
Let $\widetilde{\mathsf V}$ be a $\mathbb{C}$-vector space containing $\mathsf V$ as a subspace.
Let $\mathsf X$ be an invertible element of $\mathrm{End}(\widetilde{\mathsf V})$ such that $\mathsf V$ is invariant under $\mathsf X+\mathsf X^{-1}$.
Suppose that the action of $\mathsf A$ on $\mathsf V$ is same as the action of $\sfa + \sfb\mathsf t(\mathsf X + \mathsf X^{-1})= \sfa + \sfb\mathsf{t}\mathsf X^{-1} + \sfc\mathsf{t}^{-1}\mathsf X$ on $\mathsf V$, where $\mathsf A$ is an element of $\mathsf\Phi$ as in Definition \ref{Def:LS.dualqHahn} and $\sfa$, $\sfb$, $\sfc$ are parameters of $\mathsf\Phi$ and $\mathsf t$ is from \eqref{scalar;t.b.c}. Then, on $\mathsf V$
\begin{equation}\label{hi(X)=vi(A)}
	\mathsf h_i(\mathsf X) = \mathsf t^i(q;q)_i(\mathsf t^{-2}\sfr q^{1-\sfd};q)_i\mathsf v_i(\mathsf A), \qquad 0 \leq i \leq \sfd,
\end{equation}
where $\mathsf h_i$ and $\mathsf v_i$ are from \eqref{LaurentPoly;h(i)} and \eqref{eq:poly.v(i)}, respectively.
Moreover, for a non-zero vector $u \in \mathsf E_0\mathsf V$,
\begin{equation}\label{hi(X)=vi(A).E*iu}
	\mathsf h_i(\mathsf X).\mathsf E^*_0u = \mathsf t^i(q;q)_i(\mathsf t^{-2}\sfr q^{1-\sfd};q)_i\mathsf E^*_iu, \qquad 0 \leq i \leq \sfd.
\end{equation}
\end{lemma}
\begin{proof}
From \eqref{eq;poly.h(i)}, we have
\begin{equation}\label{pf:eq.hi(X)=fi(A)}
	\mathsf h_i(\mathsf X) = \frac{(\sfr q;q)_i(q^{-\sfd};q)_i}{\mathsf t^i}\mathsf f_i(\mathsf A),
\end{equation}
on $\mathsf V$.
Evaluate $\mathsf f_i(\mathsf A)$ in \eqref{pf:eq.hi(X)=fi(A)} using \eqref{poly:f}, \eqref{IntNum(dualqHahn)} and simplify the result to get \eqref{hi(X)=vi(A)}. 
To obtain \eqref{hi(X)=vi(A).E*iu}, use \eqref{vi(A)E*i} and \eqref{hi(X)=vi(A)}.
\end{proof}

We finish this section with a comment. 
With reference to $\mathsf \Phi$, we define the scalars
\begin{equation}\label{eq:scalar m(i)}
	\mathsf m_i=\mathrm{trace}(\mathsf E_i \mathsf E^*_0), \qquad 0 \leq i \leq \sfd.
\end{equation}
By \cite[Theorem 17.12]{2004TerLAA}, the $\mathsf m_i$, $0 \leq i \leq \sfd$, are given in terms of the parameter array of $\mathsf \Phi$ by
\begin{equation}\label{formula.mi:PA}
	\mathsf m_i = \frac{\myvarphi_1\myvarphi_2\cdots\myvarphi_i \myphi_1\myphi_2\cdots \myphi_{\sfd-i}}{(\mytheta^*_0-\mytheta^*_1)\cdots(\mytheta^*_0-\mytheta^*_\sfd)(\mytheta_i-\mytheta_0)\cdots (\mytheta_i-\mytheta_{i-1})(\mytheta_i-\mytheta_{i+1})\cdots(\mytheta_i-\mytheta_\sfd)}.
\end{equation}
Applying the formulas in Definition \ref{Def:LS.dualqHahn} with \eqref{scalar;t.b.c} to \eqref{formula.mi:PA},  the $\mathsf m_i$ are given by
\begin{equation}\label{m_formula}
	\mathsf m_i = 
	q^{\sfd-i}\mathsf r^{\sfd-i} 
	\frac{(q^{i+1};q)_{\sfd-i} (\mathsf r^{-1}\mathsf t^2q^i;q)_{\sfd-i}(\sfr q;q)_i(1-\mathsf t^2q^{2i})}{(q;q)_{\sfd-i}(\mathsf t^2q^i;q)_{\sfd+1}}.
\end{equation}

\section{The primary $\mathbf T$-module $\mathbf{W}$}\label{Section:primary T-module}

Recall the primary $\mathbf{T}$-module $\mathbf{W}$ of $\Gamma$ from Section \ref{Section:algebraT}. 
In this section, we treat four dual $q$-Hahn Leonard systems that naturally arise from the structure of $\mathbf{W}$.
Since $\mathbf{W}$ is a module for both $T$ and $\widetilde T$, it contains both $M\hat{x}$ (as a $T$-module) and $M\hat{C}$ (as a $\widetilde T$-module).
Let $M\hat{x}^\perp$ (resp. $M\hat{C}^\perp$) denote the orthogonal complement of $M\hat{x}$ (resp. $M\hat{C}$) in $\mathbf{W}$.
Note that $M\hat{x}^\perp$ is an irreducible $T$-submodule of $\mathbf{W}$ with dimension $D-1$  and $M\hat{C}^\perp$ is an irreducible $\widetilde T$-submodule of $\mathbf{W}$ with dimension $D$; cf. \cite[Sections 6, 7]{2013LeeLAA}. Therefore, $\mathbf{W}$ decomposes in two ways:
\begin{align}
	&&\mathbf{W} & = M\hat{x} \oplus M\hat{x}^\perp && \text{(orthogonal direct sum of irreducible $T$-modules)} && \label{ODS:Mx,Mxp}\\
	&&& = M\hat{C} \oplus M\hat{C}^\perp && \text{(orthogonal direct sum of irreducible $\widetilde T$-modules)}\label{ODS:MC,MCp}.
\end{align}
For the rest of the paper, we set a non-zero scalar
\begin{equation}\label{scalar;tau}
	\tau=-q^{(-N-1)/2}.
\end{equation}
Indeed, it turns out that $\tau^2=b^{-1}c$, where $b,c$ are from Proposition \ref{prop:LSonMx,Mxp}(I); cf. \eqref{scalar;t.b.c}.

\begin{proposition}\label{prop:LSonMx,Mxp}
Recall the matrices $A$, $A^*$,  $\{E_i\}^D_{i=0}$, $\{E^*_i\}^D_{i=0}$ in $T$ and the irreducible $T$-submodules $M\hat{x}$, $M\hat{x}^\perp$ of $\mathbf{W}$ from \eqref{ODS:Mx,Mxp}.
Define the following sequences of matrices by
\begin{equation*}
	\Phi:=(A, A^*,\{E_i\}^D_{i=0}, \{E^*_i\}^D_{i=0}) |_{M\hat{x}}, \quad
	\Phi^\perp:=(A, A^*,\{E_i\}^{D-1}_{i=1}, \{E^*_i\}^{D-1}_{i=1}) |_{M\hat{x}^\perp}, 
\end{equation*}
where $|_{Z}$ means that each of the matrices in the sequence is restricted to the subspace $Z$ of $\mathbf{W}$. 
The following {\rm (I)}, {\rm (II)} hold.
\begin{itemize}
\item[\rm{(I)}]
The sequence $\Phi$ is a Leonard system on  $M\hat{x}$  that has dual $q$-Hahn type. 
The parameter sequence of $\Phi$ is
$
	(a, a^*, b, b^*,c,r; q, D),
$
where
\begin{align*}
	& a = \frac{q-q^{N-D+1}-q^{D+1}-1}{(q-1)^2}, && a^* = \frac{(q^N-q)(2-q^D-q^{N-D})}{(q-1)(q^D-1)(q^{N-D}-1)},\\
	& b=\frac{q^{N+1}}{(q-1)^2}, && b^*=\frac{(q^N-q)(q^N-1)}{(q-1)(q^D-1)(q^{N-D}-1)}, \\
	& c=\frac{1}{(q-1)^2}, &&  r = q^{D-N-1}.
\end{align*}
Moreover, for $0 \leq i \leq D$, the folloinwg (i)--(iii) hold.
\begin{itemize}
	\item[\rm (i)] The vectors $A_i\hat{x}$ $(=\hat{C}^+_{i-1}+\hat{C}^-_i)$ form a $\Phi$-standard basis for $M\hat{x}$.
	\item[\rm (ii)] The intersection numbers $b_i$, $c_i$ of $\Phi$ are given by
	\begin{equation*}
	b_i  = q^{2i+1}\gauss{D-i}{1}\gauss{N-D-i}{1}, \qquad  c_i  = \gauss{i}{1}^2. \qquad
	\end{equation*} 
	\item[\rm (iii)] The monic dual $q$-Hahn polynomials $h_i$ associated with $\Phi$ (cf. \eqref{LaurentPoly;h(i)}) are given by
	\begin{equation}\label{prop:eq:h(i)}
		h_i(\zeta) = h_i(\zeta; b, c, r, D;q)=
		\tau^i(q;q)^2_i v_i(a + b\tau \zeta^{-1} + c\tau^{-1}\zeta), \qquad
	\end{equation}
	where $v_i$ are the polynomials associated with $\Phi$ as in \eqref{eq:poly.v(i)}.
\end{itemize}

\item[\rm{(II)}]
The sequence $\Phi^\perp$ is a Leonard system on $M\hat{x}^\perp$ that has dual $q$-Hahn type. 
The parameter sequence of $\Phi^\perp$ is
$
	(a^\perp, a^{*\perp}, b^\perp, b^{*\perp},c^\perp,r^\perp; q, D-2),
$
where
\begin{equation*}
	(a^\perp, a^{*\perp}, b^\perp, b^{*\perp},c^\perp,r^\perp)=(a, a^*, bq^{-1}, b^*q^{-1}, cq, rq).
\end{equation*}
Moreover, for $0 \leq i \leq D-2$, the folloinwg (i)--(iii) hold.
\begin{itemize}
	\item[\rm (i)] 
		The vectors 
	\begin{equation}\label{prop:eq:u-perp(i)}
		u^\perp_i = (q^{D-i-1}-1)\hat{C}^+_i + (q^{-i-1}-1)\hat{C}^-_{i+1}  \qquad
	\end{equation}
	form a $\Phi^\perp$-standard basis for $M\hat{x}^\perp$.
	\item[\rm (ii)] 
	The intersection numbers $b^\perp_i$, $c^\perp_i$ of $\Phi^\perp$ are given by
	\begin{equation*}
	b^\perp_i  = q^{2i+3}\gauss{D-i-2}{1}\gauss{N-D-i-1}{1}, \qquad  c^\perp_i  = q\gauss{i}{1}\gauss{i+1}{1}. \qquad
	\end{equation*}
	\item[\rm (iii)] The monic dual $q$-Hahn polynomials $h^\perp_i$ associated with $\Phi^\perp$ are given by
	\begin{equation}\label{prop:eq:h-perp(i)}
		h^\perp_i(\zeta) = h^\perp_i(\zeta; bq^{-1}, cq, rq, D-2;q) = 
		\tau^iq^i(q;q)_i(q^2;q)_i v^\perp_i(a + b\tau \zeta^{-1} + c\tau^{-1}\zeta), \qquad
	\end{equation}
	where $v^\perp_i$ are the polynomials associated with $\Phi^\perp$ as in \eqref{eq:poly.v(i)}.
\end{itemize}
\end{itemize}
\end{proposition}
\begin{proof} 
(I): Refer to \cite[Section 6]{2013LeeLAA} (or \cite[Section 6]{2018LeeTanakaSIGMA}).
Parts (i) and (ii) routinely follows.
For (iii), evaluate \eqref{eq;poly.h(i)} using \eqref{poly:f}, part (ii), and the parameter sequence of $\Phi$.\\
(II): Similar. 
\end{proof}

\begin{remark}
We note that for each irreducible $T$-module $W$ the restrictions of $A$ and $A^*$ on $W$ induce a Leonard system of dual $q$-Hahn type; cf. \cite[Theorem 4.6]{2015GaGaHoLAA}.
\end{remark}

\begin{proposition}\label{prop:LSonMC,MCp}
Recall the matrices $A$, $\widetilde A^*$,  $\{E_i\}^D_{i=0}$, $\{\widetilde E^*_i\}^{D-1}_{i=0}$ in $\widetilde T$ and the irreducible $\widetilde T$-submodules $M\hat{C}$, $M\hat{C}^\perp$ of $\mathbf{W}$ from \eqref{ODS:MC,MCp}.
Define the following sequences of matrices by
\begin{equation*}
	\widetilde\Phi:=(A, \widetilde A^*,\{E_i\}^{D-1}_{i=0}, \{\widetilde E^*_i\}^{D-1}_{i=0})|_{M\hat{C}}, \quad
	\widetilde\Phi^\perp:=(A, A^*,\{E_i\}^{D}_{i=1}, \{\widetilde E^*_i\}^{D-1}_{i=0})|_{M\hat{C}^\perp},
\end{equation*}
where $|_{Z}$ means that each of the matrices in the sequence is restricted to the subspace $Z$ of $\mathbf{W}$. 
Recall the parameter sequence $(a, a^*, b, b^*,c,r; q, D)$ of $\Phi$ from Proposition \ref{prop:LSonMx,Mxp}.
The following {\rm (I)}, {\rm (II)} hold.
\begin{itemize}
\item[\rm{(I)}]
The sequence $\widetilde\Phi$ is a Leonard system on $M\hat{C}$ that has dual $q$-Hahn type. 
The parameter sequence of $\widetilde\Phi$ is
$
	(\widetilde a, \widetilde a^{*}, \widetilde b, \widetilde b^{*}, \widetilde c, \widetilde r; q, D-1),
$
where \footnote{See \cite[Proposition 4.6]{2011TanakaE-JC} for a general result.}
\begin{equation*}
	(\widetilde a, \widetilde a^{*}, \widetilde b, \widetilde b^{*}, \widetilde c, \widetilde r) = \left( a, a^*+\frac{q-1}{q^{N-D+1}-1}b^*, b, \frac{q^{N-D}-1}{q^{N-D+1}-1}b^*, c, r  \right).
\end{equation*}
Moreover, $0 \leq i \leq D-1$, the folloinwg (i)--(iii) hold.
\begin{itemize}
	\item[\rm (i)] The vectors $\hat{C}_i$ $(=\hat{C}^-_i+\hat{C}^+_i)$ form a $\widetilde\Phi$-standard basis for $M\hat{C}$.
	\item[\rm (ii)] The intersection numbers $\widetilde b_i$, $\widetilde c_i$ of $\widetilde \Phi$ are given by
	\begin{equation*}
	 \widetilde{b}_i = q^{2i+2}\gauss{D-i-1}{1}\gauss{N-D-i}{1}, \qquad 
	\widetilde{c}_i = \gauss{i+1}{1}\gauss{i}{1}. \qquad
	\end{equation*}
	\item[\rm (iii)] The monic dual $q$-Hahn polynomials $\widetilde h_i$ associated with $\widetilde \Phi$ are given by
	\begin{equation}\label{prop:eq:wt-h(i)}
		\widetilde h_i(\zeta) = \widetilde h_i(\zeta;  b, c, r, D-1;q)=
		\tau^i(q;q)_i(q^2;q)_i\widetilde v_i(a + b\tau \zeta^{-1} + c\tau^{-1}\zeta), \qquad 
	\end{equation}
	where $\widetilde v_i$ are the polynomials associated with $\widetilde \Phi$ as in \eqref{eq:poly.v(i)}.
\end{itemize}

\item[\rm{(II)}]
The matrices of $\widetilde \Phi^\perp$ act on $M\hat{C}^\perp$ as a Leonard system that has dual $q$-Hahn type. 
The parameter sequence of $\widetilde \Phi^\perp$ is
$
	(\widetilde a^\perp, \widetilde a^{*\perp}, \widetilde b^\perp, \widetilde b^{*\perp}, \widetilde c^\perp, \widetilde r^\perp; q, D-1),
$
where
\begin{equation*}
	(\widetilde a^\perp, \widetilde a^{*\perp}, \widetilde b^\perp, \widetilde b^{*\perp}, \widetilde c^\perp, \widetilde r^\perp)=\left( a, a^*+\frac{q-1}{q^{N-D+1}-1}b^*, bq^{-1}, \frac{q^{N-D}-1}{q^{N-D+1}-1}b^*, cq, rq  \right).
\end{equation*}
Moreover, $0 \leq i \leq D-1$, the folloinwg (i)--(iii) hold.
\begin{itemize}
	\item[\rm (i)] The vectors 
	\begin{equation}\label{prop:eq:wt-u-perp(i)}
		\widetilde u^\perp_i = (q^{N-D-i}-1)\hat{C}^-_i + (q^{-i-1}-1)\hat{C}^+_{i} \qquad
	\end{equation}
	form a $\widetilde\Phi^\perp$-standard basis for $M\hat{C}^\perp$.
	\item[\rm (ii)] The intersection numbers $\widetilde b^\perp_i$, $\widetilde c^\perp_i$ of $\widetilde \Phi^\perp$ are given by
	\begin{equation*}
	\widetilde b^\perp_i  = q^{2i+2}\gauss{D-i-1}{1}\gauss{N-D-i-1}{1}, \qquad  c^\perp_i  = q\gauss{i}{1}^2. \qquad 
	\end{equation*}
	\item[\rm (iii)] The monic dual $q$-Hahn polynomials $\widetilde h^\perp_i$ associated with $\widetilde \Phi^\perp$ are given by
	\begin{equation}\label{prop:eq:wt-h-perp(i)}
		\widetilde h^\perp_i(\zeta) = \widetilde h^\perp_i(\zeta; bq^{-1},cq, rq, D-1;q)=
		\tau^iq^i(q;q)^2_i\widetilde v^\perp_i(a+b\tau\zeta^{-1}+c\tau^{-1}\zeta),
	\end{equation}
	where $\widetilde v^\perp_i$ are the polynomials associated with $\widetilde \Phi^\perp$ as in \eqref{eq:poly.v(i)}.
\end{itemize}
\end{itemize}
\end{proposition}
\begin{proof} 
Similar to Proposition \ref{prop:LSonMx,Mxp}.
\end{proof}

We comment on the decompositions \eqref{ODS:Mx,Mxp} and \eqref{ODS:MC,MCp} of $\mathbf{W}$.
We first consider the orthogonal direct sum of $\mathbf{W}$ from \eqref{ODS:Mx,Mxp}.
Let $\pi\in\mathrm{End}(\mathbf{W})$ be the orthogonal projection onto $M\hat{x}$, i.e., the element $\pi$ satisfies $(\pi-1)M\hat{x}=0$ and $\pi(M\hat{x}^\perp)=0$.
We give an action of $\pi$ on $\mathbf{W}$ as follows.
Consider a $\Phi$-standard basis $\{A_i\hat{x}\}^D_{i=0}$ for $M\hat{x}$ and a $\Phi^\perp$-standard basis $\{u^\perp_i\}^{D-2}_{i=0}$ for $M\hat{x}^\perp$. 
From them, we find that
\begin{align}
	& \hat{C}^+_{i-1} = \frac{q^i-1}{q^D-1}A_i\hat{x} + \frac{q^i}{q^D-1}u^\perp_{i-1}, \label{eq(1):C+-;Aix,upi}\\
	& \hat{C}^-_i = \frac{q^D-q^i}{q^D-1}A_i\hat{x} + \frac{q^i}{1-q^D} u^\perp_{i-1}, \label{eq(2):C+-;Aix,upi}
\end{align}
 for $1 \leq i \leq D-1$.
From \eqref{eq(1):C+-;Aix,upi} and \eqref{eq(2):C+-;Aix,upi}, the action of $\pi$ on $\hat{C}^\pm_j$ is given by
\begin{equation}\label{eq:pi.Cpm(i)}
	\pi.\hat{C}^+_{i-1} = \frac{q^i-1}{q^D-1}(\hat{C}^+_{i-1}+\hat{C}^-_i),	 
	\qquad
 	\pi.\hat{C}^-_i = \frac{q^D-q^i}{q^D-1}(\hat{C}^+_{i-1}+\hat{C}^-_i),
\end{equation}
for $1 \leq i \leq D-1$. Moreover, we have $\pi.\hat{C}^-_0 = \hat{C}^-_0$ and $\pi.\hat{C}^+_{D-1} = \hat{C}^+_{D-1}$.

Next, we consider the orthogonal direct sum of $\mathbf{W}$ from \eqref{ODS:MC,MCp}.
Let $\widetilde\pi\in\mathrm{End}(\mathbf{W})$ be the orthogonal projection onto $M\hat{C}$, i.e., the element $\widetilde\pi$ satisfies $(\widetilde\pi-1)M\hat{C}=0$ and $\widetilde\pi(M\hat{C}^\perp)=0$.
We give an action of $\widetilde\pi$ on $\mathbf{W}$ as follows.
Consider a $\widetilde\Phi$-standard basis $\{\hat{C}_i\}^{D-1}_{i=0}$ for $M\hat{C}$ and a $\widetilde\Phi^\perp$-standard basis $\{\widetilde u^\perp_i\}^{D-1}_{i=0}$ for $\widetilde M\hat{x}^\perp$. 
From them, we find that
\begin{align}
	& \hat{C}^-_i = \frac{q^{i+1}-1}{q^{N-D+1}-1}\hat{C}_i + \frac{q^{i+1}}{q^{N-D+1}-1}\widetilde u^\perp_i,\label{eq(1):C+-;Ci,tilde upi} \\
	& \hat{C}^+_i = \frac{q^{N-D+1}-q^{i+1}}{q^{N-D+1}-1}\hat{C}_i - \frac{q^{i+1}}{q^{N-D+1}-1}\widetilde u^\perp_i, \label{eq(2):C+-;Ci,tilde upi}
\end{align}
for $0 \leq i \leq D-1$.
From \eqref{eq(1):C+-;Ci,tilde upi} and \eqref{eq(2):C+-;Ci,tilde upi}, the action of $\widetilde \pi$ on $\hat{C}^\pm_j$ is given by
\begin{equation}\label{eq:wtpi.Cpm}
	\widetilde\pi(\hat{C}^-_i) = \frac{q^{i+1}-1}{q^{N-D+1}-1}(\hat{C}^-_i+\hat{C}^+_i), 
	\qquad
	\widetilde\pi(\hat{C}^+_i) = \frac{q^{N-D+1}-q^{i+1}}{q^{N-D+1}-1}(\hat{C}^-_i+\hat{C}^+_i),
\end{equation}
for $0 \leq i \leq D-1$.

\section{The confluent Cherednik algebra $\mathcal{H}_{\mathrm V}$}\label{Section:Hv}

The double affine Hecke algebra (DAHA), or Cherednik algebra, for a reduced affine root system was defined by Cherednik \cite{1992Cherednik}, and the definition was extended to non-reduced affine root systems of type $(C^\vee_n, C_n)$ by Sahi \cite{1999SahiAnnMath}.
In \cite{2016Mazzocco} Mazzocco introduced seven new algebras as degenerations of the DAHAs of type $(C^\vee_1, C_1)$ and established a new relation between the theory of the Painlev\'e equations and the theory of the $q$-Askey scheme. 
Among the seven algebras, the confluent Cherednik algebra $\mathcal{H}_{\mathrm{III}}$ \cite[(3.86)--(3.91)]{2016Mazzocco} has been shown to be recognized to a certain nil-DAHA of type $(C^\vee_1, C_1)$, which is associated with dual polar graphs; cf. \cite[Remark 8.4]{2018LeeTanakaSIGMA}.\footnote{In \cite[Definition 8.1]{2018LeeTanakaSIGMA} we overlooked the relation $\mathcal{U}'=q\mathcal{U}\mathcal{X}$ which is obtained by applying double-dot normalization to $\mathcal{T}_1=\mathcal{T}'^{-1}\mathcal{X}^{-1}$.
}
In the present paper, we shall focus our attention on the algebra $\mathcal{H}_\mathrm{V}$, another confluence Cherednik algebra among the seven,  and discuss how $\mathcal{H}_\mathrm{V}$ is related to our Grassmann graph $\Gamma$.
\begin{definition}[{\cite[Theorem 3.2.(3.73)--(3.78)]{2016Mazzocco}}]\label{Def:CheAlgHv}
Let $k,k',u,q$ be non-zero scalars in $\mathbb{C}$.
The algebra $\mathcal{H}_{\mathrm V}=\mathcal{H}_{\mathrm V}(k,k',u;q)$ is the associative $\mathbb{C}$-algebra with generators $\mathcal{T}, \mathcal{T}'$, $\mathcal{U}$, $\mathcal{U}'$ and relations
\begin{align}
	&&&(\mathcal{T}-k)(\mathcal{T}+k^{-1})=0, 	& & \mathcal{U}'(\mathcal{U}'+1)=0, && \label{Def:algHv.rel(1)} \\	
	&&&(\mathcal{T}'-k')(\mathcal{T}'+k'^{-1})=0, 	& & \mathcal{U}(\mathcal{U}+u^{-1})=0, && \label{Def:algHv.rel(2)}\\	
	&&&  q^{1/2}\mathcal{T}'\mathcal{T}\mathcal{U}' = \mathcal{U} + u^{-1}, 
	&&  q^{1/2}\mathcal{U}\mathcal{T}'\mathcal{T} = \mathcal{U}' + 1. \label{Def:algHv.rel(3)}
\end{align}
We remark that $\mathcal{T}$, $\mathcal{T}'$ are invertible.
\end{definition}

We now construct an $\mathcal{H}_\mathrm{V}$-module structure.
Recall the prime power $q$.
In what follows, we set 
\begin{equation}\label{scalars:k,k',u}
	k=\sqrt{-1}q^{(D-N-1)/2}, \qquad k' = \sqrt{-1}q^{-D/2}, \qquad u=q^{D-\frac{N}{2}}.
\end{equation}
Using these scalars, define the matrices as follows: for $0 \leq i \leq D-1$
\begin{equation}\label{matrices:t(i),u'(i)}
	t(i) := k\begin{bmatrix}
	1-q^{i+1}+q^{N-D+1} & q^{N-D+1}(q^{D-N+i}-1) \\
	1-q^{i+1} & q^{i+1}
	\end{bmatrix},
	\qquad
	u'(i) := \begin{bmatrix}
	-1 & 0 \\
	q^{-i-1} & 0
	\end{bmatrix},
\end{equation}
and for $1 \leq i \leq D-1$,
\begin{equation}\label{matrices:t'(i),u(i)}
	t'(i) := k'\begin{bmatrix}
	1-q^i+q^D & q^i-q^D \\
	1-q^i & q^i
	\end{bmatrix},
	\qquad 
	u(i) := u^{-1}\begin{bmatrix}
	-1 & 1-q^{D-i} \\
	0 & 0
	\end{bmatrix}.
\end{equation}
Moreover, define
\begin{equation}\label{matrices:t'(0,D),u(0,D)}
	t'(0) := \begin{bmatrix} k'  \end{bmatrix}, \quad
	t'(D) := \begin{bmatrix} k'  \end{bmatrix}, \quad
	u(0) := \begin{bmatrix} 0  \end{bmatrix}, \quad
	u(D) :=\begin{bmatrix} -u^{-1}  \end{bmatrix}. 
\end{equation}

\begin{lemma}\label{lem:Hv.rels}
The following (i), (ii) hold.
\begin{itemize}
\item[(i)] Let $0 \leq i \leq D-1$. 
Then $\mathrm{trace}(t(i)) = k-k^{-1}$ and $\det(t(i)) = -1$. Moreover,
\begin{equation*}
	(t(i)-k)(t(i)+k^{-1})=0, \qquad u'(i)(u'(i)+1) = 0.
\end{equation*}
\item[(ii)] Let $1 \leq i \leq D-1$.
Then $\mathrm{trace}(t'(i)) = k'-k'^{-1}$ and $\det(t'(i)) = -1$. Moreover,
\begin{equation*}
	(t'(i)-k')(t'(i)+k'^{-1})=0, \qquad u(i)(u(i)+u^{-1}) = 0. 
\end{equation*}
\end{itemize}
\end{lemma}
\begin{proof}
	Follows from \eqref{matrices:t(i),u'(i)} and \eqref{matrices:t'(i),u(i)}.
\end{proof}
Using the matrices \eqref{matrices:t(i),u'(i)}--\eqref{matrices:t'(0,D),u(0,D)}, define the block diagonal matrices in $\mathbb{C}^{2D\times 2D}$:
\begin{align*}
	&\pmb{T} := \mathrm{blockdiag}\Big[ t(0), t(1), \ldots , t(D-1) \Big], 
	&&\pmb{T}' := \mathrm{blockdiag}\Big[ t'(0), t'(1), \ldots , t'(D-1), t'(D) \Big], \\
	& \pmb{U}' := \mathrm{blockdiag}\Big[ u'(0), u'(1), \ldots , u'(D-1) \Big],
	&&{\pmb U} := \mathrm{blockdiag}\Big[ u(0), u(1), \ldots , u(D-1), u(D) \Big].
\end{align*}

\begin{lemma}\label{lem:mat.A*,wtA*}
We have both
\begin{align*}
	qk (q^{-1/2}\pmb{U}\pmb{T}' + \pmb{T}\pmb{U}') & = \mathrm{diag}\Big( 1, q^{-1}, q^{-1}, q^{-2}, q^{-2}, \ldots, q^{-(D-1)}, q^{-(D-1)}, q^{-D} \Big), \\
	qk (q^{1/2}\pmb{U}\pmb{T}' + \pmb{T}\pmb{U}') & = \mathrm{diag}\Big( 1, 1, q^{-1}, q^{-1}, q^{-2}, q^{-2}, \ldots, q^{-(D-1)}, q^{-(D-1)} \Big).
\end{align*}
\end{lemma}
\begin{proof}
Use \eqref{matrices:t(i),u'(i)}--\eqref{matrices:t'(0,D),u(0,D)}.
The result routinely follows.
\end{proof}

\begin{proposition}\label{prop:alg-homo:HtoC(2D)}
	There exists a $\mathbb{C}$-algebra homomorphism from $\mathcal{H}_{\mathrm V}$ to the full matrix algebra $\mathbb{C}^{2D\times 2D}$ that sends
\begin{equation*}
	\mathcal{T} \mapsto \pmb T, \qquad \mathcal{T}' \mapsto \pmb{T}', \qquad \mathcal{U} \mapsto \pmb{U}, \qquad \mathcal{U}' \mapsto \pmb{U}'.
\end{equation*}
\end{proposition}
\begin{proof}
The matrices $\pmb T$, $\pmb T'$, $\pmb U$, $\pmb U'$ satisfy the defining relations \eqref{Def:algHv.rel(1)}--\eqref{Def:algHv.rel(3)} by Lemma \ref{lem:Hv.rels}. The result follows.
\end{proof}

\noindent
Recall the $2D$-dimensional subspace $\mathbf{W}$ and its ordered basis $\mathcal{C}$ (cf. \eqref{Basis.W}) from Section \ref{Section:algebraT}.

\begin{corollary}\label{cor:Hv-moduleW}
There exists an $\mathcal{H}_{\mathrm{V}}$-module structure on $\mathbf{W}$ such that the matrices representing $\mathcal{T}$, $\mathcal{T}'$, $\mathcal{U}$, $\mathcal{U}'$ with respect to the ordered basis $\mathcal{C}$ are $\pmb{T}$, $\pmb{T}'$, $\pmb{U}$, $\pmb{U}'$, respectively.
\end{corollary}
\begin{proof}
Identifying $\mathrm{End}(\mathbf{W})$ with $\mathbb{C}^{2D\times 2D}$, we obtain a representation of $\mathcal{H}_\mathrm{V}$ on $\mathbf{W}$ by Proposition \ref{prop:alg-homo:HtoC(2D)}.
The result follows.
\end{proof}
By the comments below \eqref{Basis.W} and Corollary \ref{cor:Hv-moduleW},  the space $\mathbf{W}$ has a module structure for both $\mathbf{T}$ and $\mathcal{H}_{\mathrm V}$.
We shall discuss how the $\mathcal{H}_{\mathrm V}$-action on $\mathbf{W}$ is related to the $\mathbf{T}$-action on $\mathbf{W}$.
Recall the scalar $\tau$ from \eqref{scalar;tau}. Observe that $\tau=kk'$.
For notational convenience, we define
\begin{equation}\label{eq:X}
	\mathcal{X} := \mathcal{T}'\mathcal{T}.
\end{equation}
Observe that $\mathcal{X}$ is invertible since $\mathcal{T}'$, $\mathcal{T}$ are invertible.
We also define 
\begin{equation*}\label{eq:A,A*,wtA*}
	\mathcal{A}  := \mathcal{X}+\mathcal{X}^{-1}, \qquad \mathcal{A}^* := qk(q^{-1/2}\mathcal{U}\mathcal{T}'+\mathcal{T}\mathcal{U}'),  \qquad \widetilde{\mathcal{A}}^* := qk(q^{1/2}\mathcal{U}\mathcal{T}'+\mathcal{T}\mathcal{U}').
\end{equation*}
We give the actions of the elements $\mathcal{A}$, $\mathcal{A}^*$, $\widetilde{\mathcal{A}}^*$ of $\mathcal{H}_\mathrm{V}$ on $\mathbf{W}$ with the basis $\mathcal{C}$.
\begin{lemma}\label{lem:actionX}
The following (i), (ii) hold.
\begin{itemize}
\item[(i)]
The actions of $\mathcal{X}$ and $\mathcal{X}^{-1}$ on $\hat{C}^-_i$, $0 \leq i \leq D-1$, are given as linear combination with the following terms and coefficients.
\begin{align*}
	& \mathcal{X}.\hat{C}^-_i : && \mathcal{X}^{-1}.\hat{C}^-_i : \\
	&\begin{tabular}{c|c}
	term & coefficient \\
	\hline \hline \\ [-0.8em]
	$\hat C^+_{i-1}$ & $\tau(q^i-q^D)(q^{N-D+1}-q^{i+1}+1) $\\[0.5em]
	$\hat C^-_{i}$ & $\tau q^i(q^{N-D+1}-q^{i+1}+1)$ \\[0.5em]
	$\hat C^+_{i}$ & $\tau (1-q^{i+1})(q^D-q^{i+1}+1)$ \\[0.5em]
	$\hat C^-_{i+1}$ & $\tau (q^{i+1}-1)^2$ \\[0.5em]
	\end{tabular} \quad , 
	&&
	\begin{tabular}{c|c}
	term & coefficient \\
	\hline \hline \\ [-0.8em]
	$\hat{C}^-_{i-1}$ & $\tau(1-q^{i-D})(q^{N+1}-q^{D+i}) $\\[0.5em]
	$\hat C^+_{i-1}$ & $\tau (q^D-q^{i})(q^{N-D+1}-q^i+1)$ \\[0.5em]
	$\hat C^-_{i}$ & $\tau q^{i+1}(q^D-q^i+1)$ \\[0.5em]
	$\hat C^+_{i}$ & $\tau (q^{i+1}-1)(q^D-q^i+1)$ \\[0.5em]
	\end{tabular}.
\end{align*} 
\item[(ii)]
The actions of $\mathcal{X}$ and $\mathcal{X}^{-1}$ on $\hat{C}^+_i$, $0 \leq i \leq D-1$, are given as linear combination with the following terms and coefficients.
\begin{align*}
	& \mathcal{X}.\hat{C}^+_i : && \mathcal{X}^{-1}.\hat{C}^+_i : \\
	& \begin{tabular}{c|c}
	term & coefficient \\
	\hline \hline \\ [-0.8em]
	$\hat C^+_{i-1}$ & $\tau q^{N+1}(q^{i-D}-1)(q^{D-N+i}-1)$\\[0.5em]
	$\hat C^-_{i}$ & $\tau q^{N-D+1+i}(q^{D-N+i}-1)$ \\[0.5em]
	$\hat C^+_{i}$ & $\tau q^{i+1}(q^D-q^{i+1}+1)$ \\[0.5em]
	$\hat C^-_{i+1}$ & $\tau q^{i+1}(1-q^{i+1})$ \\[0.5em]
	\end{tabular} \quad , 
	&&\begin{tabular}{c|c}
	term & coefficient \\
	\hline \hline \\ [-0.8em]
	$\hat C^-_{i}$ & $\tau q^{N-D+2+i}(1-q^{D-N+i})$\\[0.5em]
	$\hat C^+_{i}$ & $\tau q^{i+1}(q^{N-D+1}-q^{i+1}+1)$ \\[0.5em]
	$\hat C^-_{i+1}$ & $\tau q^{i+2}(q^{i+1}-1)$ \\[0.5em]
	$\hat C^+_{i+1}$ & $\tau (q^{i+1}-1)(q^{i+2}-1)$ \\[0.5em]
	\end{tabular}.
\end{align*} 
\end{itemize}
\end{lemma}
\begin{proof}
Routine using \eqref{matrices:t(i),u'(i)}--\eqref{matrices:t'(0,D),u(0,D)} and Corollary \ref{cor:Hv-moduleW}.
\end{proof}

\begin{lemma}\label{lem:actionA}
The following (i), (ii) hold.
\begin{itemize}
\item[(i)] The action of $\mathcal{A}$ on $\hat{C}^{\pm}_i$, $0 \leq i \leq D-1$, is given as linear combination with the following terms and coefficients.
\begin{align*}
	& \mathcal{A}.\hat{C}^-_i : && \mathcal{A}.\hat{C}^+_i : \\
	& \begin{tabular}{c|c}
	term & coefficient \\
	\hline \hline \\ [-0.8em]
	$\hat C^-_{i-1}$ & $\tau q^{2i}(q^{D-i}-1)(q^{N-D-i+1}-1)$\\[0.5em]
	$\hat C^+_{i-1}$ & $\tau q^{2i}(q-1)(q^{D-i}-1)$\\[0.5em]
	$\hat C^-_{i}$ & $\tau q^i(q^{N-D+1}+q^{D+1}-2q^{i+1}+q+1)$ \\[0.5em]
	$\hat C^+_{i}$ & $\tau q^i(q-1)(q^{i+1}-1)$ \\[0.5em]
	$\hat C^-_{i+1}$ & $\tau (q^{i+1}-1)^2$ \\[0.5em]
	\end{tabular} \quad , 
	&& \begin{tabular}{c|c}
	term & coefficient \\
	\hline \hline \\ [-0.8em]
	$\hat C^+_{i-1}$ & $\tau q^{2i+1}(q^{D-i}-1)(q^{N-D-i}-1)$\\[0.5em]
	$\hat C^-_{i}$ & $\tau q^{2i+1}(q-1)(q^{N-D-i}-1)$\\[0.5em]
	$\hat C^+_{i}$ & $\tau q^{i+1}(q^{N-D+1}+q^D-2q^{i+1}+2)$ \\[0.5em]
	$\hat C^-_{i+1}$ & $\tau q^{i+1}(q-1)(q^{i+1}-1)$ \\[0.5em]
	$\hat C^+_{i+1}$ & $\tau (q^{i+1}-1)(q^{i+2}-1)$ \\[0.5em]
	\end{tabular}.
\end{align*} 

\item[(ii)] The actions of $\mathcal{A}^*$ and $\widetilde{\mathcal{A}}^{*}$ on $\hat{C}^\pm_i$, $0 \leq i \leq D-1$, are given by
\begin{align*}
	&&& \mathcal{A}^*.\hat{C}^-_i = q^{-i}\hat{C}^-_i, && \mathcal{A}^*.\hat{C}^+_i = q^{-i-1}\hat{C}^+_i, &&\\
	&&& \widetilde{\mathcal{A}}^*.\hat{C}^-_i = q^{-i}\hat{C}^-_i, && \widetilde{\mathcal{A}}^*.\hat{C}^+_i = q^{-i}\hat{C}^+_i. 
\end{align*}
\end{itemize}
\end{lemma}
\begin{proof}
(i): Routine using Lemma \ref{lem:actionX}. \\
(ii): By Lemma \ref{lem:mat.A*,wtA*}.
\end{proof}

\begin{theorem}\label{thm:1stResult}
Recall the generators $A$, $A^*$, $\widetilde A^*$ of $\mathbf{T}$ and the elements $\mathcal{A}$, $\mathcal{A}^*$, $\widetilde{\mathcal{A}}^*$ of $\mathcal{H}_\mathrm{V}$.
On $\mathbf{W}$, we have
\begin{align}
	&A = \tau b \mathcal{A} + a, \label{thm;eq(1);A,A*,wtA*} \\
	&A^*  = b^* \mathcal{A}^* + a^*, \label{thm;eq(2);A,A*,wtA*} \\
	&\widetilde A^*  = \widetilde b^* \widetilde{\mathcal{A}}^* + \widetilde a^*,\label{thm;eq(3);A,A*,wtA*} 
\end{align}
where $\tau$ is from \eqref{scalar;tau} and $a$,  $a^*$, $\widetilde a^*$, $b$, $b^*$, $\widetilde b^*$ are from Propositions \ref{prop:LSonMx,Mxp}(I) and \ref{prop:LSonMC,MCp}(I).
\end{theorem}
\begin{proof}
The identity \eqref{thm;eq(1);A,A*,wtA*} follows from Lemma \ref{lem:actionAdjMatA} and Lemma \ref{lem:actionA}(i). 
The identities \eqref{thm;eq(2);A,A*,wtA*} and  \eqref{thm;eq(3);A,A*,wtA*} follow from Lemma \ref{lem:actionMatA*,wtA*} and Lemma \ref{lem:actionA}(ii). 
\end{proof}

\begin{remark}
(i) By Theorem \ref{thm:1stResult} and since $\mathbf{W}$ is irreducible as a $\mathbf{T}$-module, it follows that an $\mathcal{H}_\mathrm{V}$-module $\mathbf{W}$ is \emph{irreducible}.

\smallskip \noindent
(ii) On the $\mathcal{H}_\mathrm{V}$-module $\mathbf{W}$, the elements $\mathcal{T}$ and $\mathcal{T}'$ are both diagonalizable.
Moreover, the element $(\mathcal{T}+k^{-1})/(k+k^{-1})$ (resp. $(\mathcal{T}'+k'^{-1})/(k'+k'^{-1})$) acts as the projection from $\mathbf{W}$ onto the eigenspace of $\mathcal{T}$ (resp. $\mathcal{T}'$) corresponding to $k$ (resp. $k'$).
\end{remark}

\begin{theorem}\label{thm:2ndResult}
Recall the orthogonal projection $\pi$ (resp. $\widetilde \pi$) from $\mathbf{W}$ onto $M\hat{x}$ (resp. $M\hat{C}$). On $\mathbf{W}$, we have
\begin{equation}\label{thm;eq;pi=T'}
	\pi = \frac{\mathcal{T}'+k'^{-1}}{k'+k'^{-1}}, \qquad \qquad
	\widetilde\pi = \frac{\mathcal{T}+k^{-1}}{k+k^{-1}}.
\end{equation}
\end{theorem}
\begin{proof}
Use \eqref{eq:pi.Cpm(i)} and the matrix $\pmb T'$ to obtain the first identity in \eqref{thm;eq;pi=T'}.
Use \eqref{eq:wtpi.Cpm} and the matrix $\pmb T$ to obtain the second identity. The result follows.
\end{proof}

We should like to make a comment on a nil-DAHA of type $(C^\vee_1, C_1)$.
We first recall the definition of the (ordinary) DAHA of type $(C^\vee_1, C_1)$.
The DAHA $\mathcal{H}=\mathcal{H}(\kappa_0, \kappa_1,\kappa_0',\kappa_1';q)$ is the associative $\mathbb{C}$-algebra with generators $\mathbf{T}^{\pm1}_0$, $\mathbf{T}^{\pm1}_1$, and $\mathbf{X}^{\pm1}$ and relations (cf. \cite[Section 6.4]{2003Macdonald}, \cite[Section 3]{1999SahiAnnMath})
\begin{equation}\label{DAHArels(1)}
	(\mathbf{T}_i-\kappa_i)(\mathbf{T}_i+\kappa_i^{-1})=0, \qquad
	(\mathbf{T}'_i-\kappa'_i)(\mathbf{T}'_i+\kappa'^{-1}_i)=0, \qquad 
	i = 0,1,
\end{equation}
where,
\begin{equation*}\label{DAHArels(2)}
	\mathbf{T}'_0 := q^{-1/2}\mathbf{X}\mathbf{T}^{-1}_0, \qquad
	\mathbf{T}'_1 := \mathbf{X}^{-1}\mathbf{T}^{-1}_1.
\end{equation*}

In \cite[Remark 8.2]{2018LeeTanakaSIGMA}, we specialized some defining relations of $\mathcal{H}$ using the so-called ``double-dot normalization'' method; cf. \cite[Section 2.5]{2015CheOrrMathZ}.
We then obtained a certain nil-DAHA, which is isomorphic to the algebra $\mathcal{H}_{\mathrm{III}}$.
In the present paper, by employing the techniques used in \cite{2018LeeTanakaSIGMA}, we shall specialize the algebra $\mathcal{H}$ to obtain a new nil-DAHA, denoted by $\overline{\mathcal{H}}$, which is well-suited in the context of Grassmann graphs. 
Set 
\begin{equation}\label{ddot:T_1,T'_1}
	\ddot{\mathbf{T}}_1 := \kappa_1\mathbf{T}_1, \qquad
	\ddot{\mathbf{T}}'_1 :=\kappa_1\mathbf{T}'_1.
\end{equation}
Apply \eqref{ddot:T_1,T'_1} to the relations \eqref{DAHArels(1)} for $i=1$ to get
\begin{equation}\label{ddot-normal;eq(2)}
	(\ddot{\mathbf{T}}_1-\kappa_1^2)(\ddot{\mathbf{T}}_1+1)=0, \qquad
	(\ddot{\mathbf{T}}'_1-\kappa_1\kappa'_1)(\ddot{\mathbf{T}}'_1+\kappa_1\kappa'^{-1}_1)=0.
\end{equation}
Observe that $\mathbf{T}'_1 = \mathbf{X}^{-1}(\mathbf{T}_1-\kappa_1+\kappa_1^{-1})$ and $\mathbf{T}_1 = (\mathbf{T}'_1-\kappa'_1+\kappa'^{-1}_1)\mathbf{X}^{-1}$.  
Use these and \eqref{ddot:T_1,T'_1} to get
\begin{equation}\label{ddot-normal;eq(1)}
	\ddot{\mathbf{T}}'_1 = \mathbf{X}^{-1}(\ddot{\mathbf{T}}_1-\kappa^2_1+1), \qquad
	\ddot{\mathbf{T}}_1 = (\ddot{\mathbf{T}}'_1-\kappa_1\kappa'_1+\kappa_1\kappa'^{-1}_1)\mathbf{X}^{-1}.
\end{equation}
Thus, $\mathcal{H}$ has a presentation with new generators $\mathbf{T}^{\pm1}_0$, $\ddot{\mathbf{T}}_1$, and $\mathbf{X}^{\pm1}$ and relations \eqref{DAHArels(1)} at $i=1$ and \eqref{ddot-normal;eq(2)} and \eqref{ddot-normal;eq(1)}.
We now specialize the parameters $\kappa_1, \kappa'_1$. 
Let $u \in \mathbb{C}$ be a nonzero scalar.
Set $\kappa'_1=u^{-1}\kappa_1$ in \eqref{ddot-normal;eq(2)} and \eqref{ddot-normal;eq(1)}.
Then, letting $\kappa_1 \to 0$, the relations \eqref{ddot-normal;eq(2)} and \eqref{ddot-normal;eq(1)} become
$$
	\ddot{\mathbf{T}}_1(\ddot{\mathbf{T}}_1+1)=0, \qquad
	\ddot{\mathbf{T}}'_1(\ddot{\mathbf{T}}'_1+u)=0, \qquad
	\ddot{\mathbf{T}}'_1=\mathbf{X}^{-1}(\ddot{\mathbf{T}}_1+1), \qquad
	\ddot{\mathbf{T}}_1=(\ddot{\mathbf{T}}'_1+u)\mathbf{X}^{-1}.
$$
Define 
$$
	\mathcal{T}:=\mathbf{T}_0, \qquad
	\mathcal{U}:=u^{-1}\ddot{\mathbf{T}}_1, \qquad
	\mathcal{X}:=q^{-1/2}\mathbf{X}, \qquad
	k:=\kappa_0, \qquad
	k':=\kappa_0'
$$
Then the algebra $\overline{\mathcal{H}}=\overline{\mathcal{H}}(k,k',u;q)$ obtained from this specialization has a presentation with generators $\mathcal{T}^{\pm1}$, $\mathcal{U}$, $\mathcal{X}^{\pm1}$ and relations
\begin{align*}
	(\mathcal{T}-k)(\mathcal{T}+k^{-1})&=0, \\
	 (\mathcal{T}'-k')(\mathcal{T}'+k'^{-1})&=0,\\
	 \mathcal{U}(\mathcal{U}+u^{-1})&=0, \\
	 \mathcal{U}'(\mathcal{U}'+1)&=0, \\
 	q^{1/2}\mathcal{U}\mathcal{X} &= \mathcal{U}'+1.
\end{align*}
where 
$$
	\mathcal{T}'=\mathcal{X}\mathcal{T}^{-1}, \qquad \qquad 
	\mathcal{U}'=q^{-1/2}\mathcal{X}^{-1}(\mathcal{U}+u^{-1}).
$$
We call $\overline{\mathcal{H}}$ a \emph{nil-DAHA} of type $(C^\vee_1, C_1)$.
We shall remark that the nil-DAHA $\overline{\mathcal{H}}$ is isomorphic to the algebra $\mathcal{H}_\mathrm{V}(k,k',u;q)$ from Definition \ref{Def:CheAlgHv}.

\section{Non-symmetric dual $q$-Hahn polynomials}\label{Section:nonsym dual q-Hahn poly}

In this section, we shall define non-symmetric dual $q$-Hahn polynomials and give them a combinatorial interpretation.
Recall the Leonard systems $\Phi$, $\Phi^\perp$, $\widetilde \Phi$, and $\widetilde \Phi^\perp$ from Propositions \ref{prop:LSonMx,Mxp} and \ref{prop:LSonMC,MCp}.
Recall the sequences of monic dual $q$-Hahn polynomials
\begin{equation*}
	\{h_i\}^D_{i=0}, \qquad \{h^\perp_i\}^{D-2}_{i=0}, \qquad \{\widetilde h_i\}^{D-1}_{i=0}, \qquad \{\widetilde h^\perp_i\}^{D-1}_{i=0}
\end{equation*}
associated with $\Phi$, $\Phi^\perp$, $\widetilde \Phi$, $\widetilde \Phi^\perp$, respectively, from \eqref{prop:eq:h(i)}, \eqref{prop:eq:h-perp(i)}, \eqref{prop:eq:wt-h(i)}, \eqref{prop:eq:wt-h-perp(i)}.
Define the following monic Laurent polynomials in $\mathbb{C}[\zeta, \zeta^{-1}]$ by
\begin{equation}\label{eq:p-perp,p-tilde,p-tildeperp}
	p^\perp = \zeta^{-1}(\zeta-\tau)(\zeta-\tau^{-1}q^{-D}), \qquad
	\widetilde p = \zeta^{-1}(\zeta-\tau^{-1}q^{-D}), \qquad
	\widetilde p^\perp = \zeta^{-1}(\zeta-\tau).
\end{equation}

\begin{lemma}\label{lem:p-perp,p-tilde,p-tildeperp}
On the $\mathcal{H}_\mathrm{V}$-module $\mathbf{W}$, we have
\begin{equation}\label{lem:eq:three-poly-action}
	p^\perp(\mathcal{X}).\hat{x} = \tau q(1-q) u_0^\perp, \qquad
	\widetilde p(\mathcal{X}).\hat{x} = (1-q) \hat{C}, \qquad
	\widetilde p^\perp(\mathcal{X}).\hat{x} = q^{D-N}\widetilde u^\perp_0,
\end{equation}
where $u^\perp_0$ is from \eqref{prop:eq:u-perp(i)} and $\widetilde u^\perp_0$ is from \eqref{prop:eq:wt-u-perp(i)}.
\end{lemma}
\begin{proof}
Note that $\hat{x}=\hat{C}^-_0$.
Setting $i=0$ in Lemma \ref{lem:actionX}(i), the actions of $\mathcal{X}$ and $\mathcal{X}^{-1}$ on $\hat{C}^-_0$ are given by
\begin{align*}
	& \mathcal{X}.\hat{x}  = \tau (q^{N-D+1}-q+1)\hat C^-_{0} + \tau (1-q)(1-q+q^D)\hat C^+_{0} + \tau (1-q)^2\hat C^-_{1}, \\
	& \mathcal{X}^{-1}.\hat{x}  = \tau q^{D+1}\hat C^-_{0} + \tau q^D(q-1)\hat C^+_{0}.
\end{align*}
Evaluate $p^\perp(\mathcal{X}).\hat{x}$ using these equations and simplify the result using \eqref{prop:eq:u-perp(i)} at $i=0$ to get the first equation in \eqref{lem:eq:three-poly-action}.
The remaining two equations in \eqref{lem:eq:three-poly-action} are similarly obtained.
\end{proof}

We define the non-symmetric Laurent polynomials $\ell^\pm_i$ in $\mathbb{C}[\zeta, \zeta^{-1}]$ as follows. For $0 \leq i \leq D-1$,
\begin{align}
	& \ell^-_i (\zeta) := \frac{q^D-q^i}{\tau^{i}(q^D-1)(q;q)^2_i}\left( h_i - \frac{1-q^i}{q^D-q^i}p^\perp h^\perp_{i-1} \right), \label{eq(1):NonSymDqHahn} \\
	& \ell^+_i (\zeta) := \frac{q^{i+1}-1}{\tau^{i+1}(q^D-1)(q;q)^2_{i+1}}\left( h_{i+1} -p^\perp h^\perp_i \right),\label{eq(2):NonSymDqHahn}
\end{align}
where 
\begin{equation}\label{hperp(D-1)}
	h^\perp_{D-1}(\zeta) := \prod^{D-1}_{j=1} (\zeta+\zeta^{-1}-\tau q^j - \tau^{-1} q^{-j}) = \zeta^{1-D}\prod^{D-1}_{j=1} (\zeta - \tau q^j)(\zeta - \tau^{-1}q^{-j}).
\end{equation}

\begin{lemma}\label{lem:h-perp(D-1)vanish}
Recall the subspace $M\hat{x}^\perp$ of $\mathbf{W}$ from \eqref{ODS:Mx,Mxp}.
Then $h^\perp_{D-1}(\mathcal{X})$ vanishes on $M\hat{x}^\perp$.
\end{lemma}
\begin{proof}
On the $\mathcal{H}_\mathrm{V}$-module $\mathbf{W}$, using \eqref{thm;eq(1);A,A*,wtA*} we find
\begin{align*}
	h^\perp_{D-1}(\mathcal{X}) = \prod^{D-1}_{j=1} (\mathcal{X}+\mathcal{X}^{-1}-\tau q^j - \tau^{-1} q^{-j}) = (\tau b)^{1-D}\prod^{D-1}_{j=1} (A-\theta_j),
\end{align*}
where we recall $\tau^2=b^{-1}c$ and $\theta_j = a + bq^{-j}+cq^{j}$.
Since $M\hat{x}^\perp=\sum^{D-1}_{j=1}E_j(M\hat{x}^\perp)$, the result follows.
\end{proof}

We define another non-symmetric Laurent polynomials $\widetilde \ell^\pm_i$ as follows. For $0 \leq i \leq D-1$,
\begin{align*}
	& \widetilde \ell^-_i (\zeta) := \frac{1}{\tau^i(1-q^{N-D+1})(q;q)^2_i}\left( \widetilde{p} \widetilde{h}_i  - q^{N-D+1} \widetilde{p}^\perp \widetilde{h}^\perp_i \right),\\
	& \widetilde \ell^+_i (\zeta) := \frac{q^{N-D+1}}{\tau^i(q^{N-D+1}-1)(q;q)^2_i}\left( \frac{1-q^{D-N+i}}{1-q^{i+1}}\widetilde{p} \widetilde{h}_i  - \widetilde{p}^\perp \widetilde{h}^\perp_i \right).
\end{align*}
We shall now give a description of the role of the Laurent polynomials $\ell^{\pm}_i$, $\widetilde \ell^\pm_i$ on $\mathbf{W}$.
Recall the basis $\mathcal{C}$ for $\mathbf{W}$ from \eqref{Basis.W}.
\begin{proposition}\label{prop:C_i=ell_i(X)x}
On the $\mathcal{H}_\mathrm{V}$-module $\mathbf{W}$, 
\begin{equation}\label{lem:eq:C_i=ell_i(X)x}
	\hat{C}^-_i = \ell^-_i(\mathcal{X}).\hat{x} = \widetilde\ell^-_i(\mathcal{X}).\hat{x}, \qquad \hat{C}^+_i = \ell^+_i(\mathcal{X}).\hat{x} = \widetilde \ell^+_i(\mathcal{X}).\hat{x},
\end{equation}
for $0 \leq i \leq D-1$.
\end{proposition}
\begin{proof}
We first show the equation $\hat{C}^+_i = \ell^+_i(\mathcal{X}).\hat{x}$, $0 \leq i \leq D-1$.
Recall the equation \eqref{eq(1):C+-;Aix,upi}.
We assume $0 \leq i \leq D-2$.
Applying \eqref{vi(A)E*i}  and \eqref{thm;eq(1);A,A*,wtA*} to each summand of the right side of \eqref{eq(1):C+-;Aix,upi} and using \eqref{prop:eq:h(i)} and \eqref{prop:eq:h-perp(i)}, we get 
\begin{equation*}\label{lem:pf:eq:C_i=ell_i(X)x}
	\hat{C}^+_i = \frac{q^{i+1}-1}{q^D-1}\frac{h_{i+1}(\mathcal{X})}{\tau^{i+1}(q;q)^2_{i+1}}\hat{x} + \frac{q^{i+1}}{q^D-1}\frac{h^\perp_i(\mathcal{X})}{\tau^iq^i(q;q)_i(q^2;q)_i}u^\perp_0.
\end{equation*}
Simplify the right side of this equation using the first equation in \eqref{lem:eq:three-poly-action}. 
We find $\hat{C}^+_i = \ell^+_i(\mathcal{X}).\hat{x}$, $0 \leq i \leq D-2$.
We now assume $i=D-1$. 
By the first equation of line \eqref{lem:eq:three-poly-action} and Lemma \ref{lem:h-perp(D-1)vanish}, it follows 
\begin{equation}\label{prop:eq:p-perp.h-perp=0}
	p^\perp(\mathcal{X})h^\perp_{D-1}(\mathcal{X}).\hat{x} = \tau q(1-q) h^\perp_{D-1}(\mathcal{X}).u^\perp_0 =0.
\end{equation}
By this comment, we find $\hat{C}^{+}_{D-1}=\ell^+_{D-1}(\mathcal{X}).\hat{x}$.
The desired result follows.
The remaining equations in \eqref{lem:eq:C_i=ell_i(X)x} are obtained in a similar way using \eqref{eq(2):C+-;Aix,upi}, \eqref{eq(1):C+-;Ci,tilde upi} and \eqref{eq(2):C+-;Ci,tilde upi}.
\end{proof}

\begin{remark}\label{rmk:degree.ell}
By \eqref{lem:eq:C_i=ell_i(X)x}, we have $\ell^-_i = \widetilde \ell^-_i$, $\ell^+_i = \widetilde \ell^+_i$, $0 \leq i \leq D-1$. From this, it follows that 
\begin{itemize}
	\itemsep-0.1em
	\item $\ell^-_i$ has the highest degree $i$ and the lowest degree is $-i$;
	\item $\ell^+_i$ has the highest degree $i$ and the lowest degree $-i-1$. 
\end{itemize}
\end{remark}

Consider the subspace $\mathcal{L}$ of $\mathbb{C}[\zeta, \zeta^{-1}]$ defined by
\begin{equation}\label{eq:space L}
	\mathcal{L} = \sum^{D-1}_{i=-D} \mathbb{C}\zeta^i.
\end{equation}
Observe that $\dim{\mathcal{L}}=2D$ and by Remark \ref{rmk:degree.ell} $\{\ell^\pm_i\}^{D-1}_{i=0}$ forms a basis for $\mathcal{L}$.
We call $\ell^\pm_i=\widetilde \ell^\pm_i$, $0 \leq i \leq D-1$, the \emph{non-symmetric dual $q$-Hahn polynomials}; see Figure \ref{NonsymDualqHahn}.
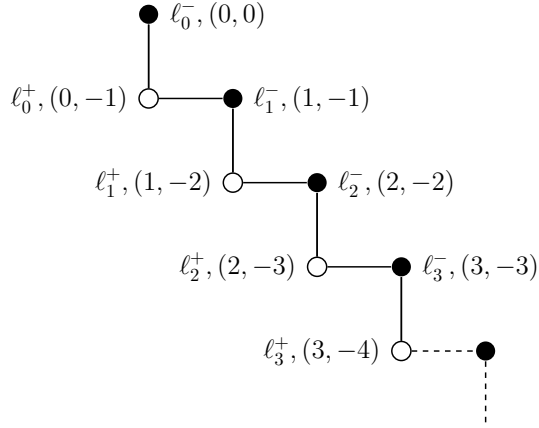
\begin{figure}
\centering
\scalemath{0.7}{
\begin{tikzpicture}
  [scale=.8,thick,auto=left, every node/.style={circle}] 
  \node[fill=black,label=right:{\Large$\ell^-_0, (0,0)$}] (n1) at (0,0) {};
  \node[draw, label=left:{\Large$\ell^+_0, (0,-1)$}] (n2) at (0,-2)  {};
  \node[fill=black,label=right:{\Large$\ell^-_1, (1,-1)$}] (n3) at (2,-2)  {};
  \node[draw,label=left:{\Large$\ell^+_1, (1,-2)$}] (n4) at (2,-4) {};
  \node[fill=black,label=right:{\Large$\ell^-_2, (2,-2)$}] (n5) at (4,-4)  {};
  \node[draw,label=left:{\Large$\ell^+_2, (2,-3)$}] (n6) at (4,-6)  {};
  \node[fill=black,label=right:{\Large$\ell^-_3, (3,-3)$}] (n7) at (6,-6)  {};
  \node[draw,label=left:{\Large$\ell^+_3, (3,-4)$}] (n8) at (6,-8)  {};
  \node[fill=black,label=right:{}] (n9) at (8,-8)  {};
  \node (n10) at (8,-10)  {};

  \foreach \from/\to in 
  {n1/n2,n2/n3,n3/n4,n4/n5,n5/n6,n6/n7,n7/n8}
      \draw (\from) -- (\to);
      \draw (n1) -- (n2) ;
      \draw (n3) -- (n4) ;
      \draw (n5) -- (n6) ;
      \draw (n7) -- (n8) ;
      \draw (n8) -- (n9) [dashed];
      \draw (n9) -- (n10) [dashed];
            
\end{tikzpicture}}
\caption{Non-symmetric dual $q$-Hahn polynomials $\ell^\pm_i$}\label{NonsymDualqHahn}
\end{figure}
In this figure, each black node (resp. white node) represents the non-symmetric dual $q$-Hahn polynomial $\ell^-_i$ (resp. $\ell^+_i$).
For each $\ell^\pm_i$, the corresponding ordered pair $(m,n)$ means that $m$ is the highest degree and $n$ is the lowest degree of $\ell^\pm_i$. 
Figure \ref{NonsymDualqHahn} shows  how  $\ell^\pm_i$ are interpreted in a combinatorial sense; cf. Figure \ref{2-dim'l partition}.
We may regard the non-symmetric dual $q$-Hahn polynomials $\ell^\pm_i$ as a discretization of non-symmetric continuous dual $q$-Hahn polynomials.
We remark that non-symmetric continuous dual $q$-Hahn polynomials were used to prove the faithfulness of a so-called basic representation on the space of Laurent polynomials in one variable for the algebra $\mathcal{H}_\mathrm{V}$; cf. \cite[Section 2]{2014Mazzocco}.

\section{Recurrence and orthogonality relations}\label{Section:rec orth relations}

We continue to discuss non-symmetric dual $q$-Hahn polynomials $\ell^\pm_i$, $0 \leq i \leq D-1$.
In this section, we derive combinatorial recurrence and orthogonality relations for $\ell^\pm_i$ from the $\mathcal{H}_V$-module $\mathbf{W}$. 
We begin with a lemma.
\begin{lemma}\label{lem:min.poly(X)}
Recall $p^\perp$ and $h^\perp_{D-1}$ from \eqref{eq:p-perp,p-tilde,p-tildeperp} and \eqref{hperp(D-1)}, respectively. 
Recall $\mathcal{X}$ from \eqref{eq:X}.
The following (i)-(iii) hold.
\begin{itemize}
\item[(i)] The element $p^\perp(\mathcal{X})h^\perp_{D-1}(\mathcal{X})$ vanishes on $\mathbf{W}$.
\item[(ii)] Let $\mu$ be the polynomial in $\mathbb{C}[\zeta]$ defined by
\begin{equation}\label{eq:min.poly(X)}
	\mu(\zeta) := \zeta^Dp^\perp h^\perp_{D-1} = (\zeta-\tau)(\zeta-\tau^{-1}q^{-D})\prod^{D-1}_{i=1} (\zeta - \tau q^i)(\zeta-\tau^{-1} q^{-i}).
\end{equation}
Then $\mu$ is the minimal polynomial of $\mathcal{X}$ on $\mathbf{W}$.
\end{itemize}
\end{lemma}
\begin{proof}
(i) Let $\varepsilon\in \{+,-\}$. By Lemma \ref{lem:h-perp(D-1)vanish}, Proposition \ref{prop:C_i=ell_i(X)x} and \eqref{prop:eq:p-perp.h-perp=0}, we have
\begin{equation*}
	p^\perp(\mathcal{X})h^\perp_{D-1}(\mathcal{X}).\hat{C}^\varepsilon_i = \ell^\varepsilon_i(\mathcal{X})h^\perp_{D-1}(\mathcal{X})p^\perp(\mathcal{X}).\hat{x} = 0, \qquad 0 \leq i \leq D-1.
\end{equation*}
(ii) Observe that $\mu$ has degree $2D$ and $\mu(\mathcal{X})=0$ on $\mathbf{W}$ by part (i). From Proposition \ref{prop:C_i=ell_i(X)x}, the result follows.
\end{proof}
Define the Laurent polynomials $\ell^\pm_{-1}$ and $\ell^\pm_D$ in $\mathbb{C}[\zeta, \zeta^{-1}]$ by $\ell^\pm_{-1}:=0$ and 
\begin{align*}
	& \ell^-_D :=\frac{1}{\tau^D(q;q)^2_D}p^\perp h^\perp_{D-1} = \frac{1}{\tau^D(q;q)^2_D}(\zeta^D+\cdots +q^{-D}\zeta^{-D}), \\
	& \ell^+_D :=\frac{\tau q^{D+1}-\zeta^{-1}}{\tau^{D+1}(q;q)_D(q;q)_{D+1}}p^\perp h^\perp_{D-1} = \frac{\tau q^{D+1}-\zeta^{-1}}{\tau^{D+1}(q;q)_D(q;q)_{D+1}}(\zeta^D+\cdots +q^{-D}\zeta^{-D}).
\end{align*}
Observe that 
\begin{equation*}
	\ell^-_D \equiv \frac{\zeta^D}{\tau^D(q;q)^2_D} \pmod{\mathcal{L}}, 
	\qquad
	\ell^+_D \equiv \frac{\tau q^{D+1}\zeta^D - q^{-D}\zeta^{-D-1}}{\tau^{D+1}(q;q)_D(q;q)_{D+1}} \pmod{\mathcal{L}}.
\end{equation*}
Moreover, by Lemma \ref{lem:min.poly(X)}(i) $\ell^\pm_D(\mathcal{X})$ vanish on $\mathbf{W}$.
By these comments and using \eqref{eq(1):NonSymDqHahn}, \eqref{eq(2):NonSymDqHahn} at $i=D-1$, we routinely find 
\begin{align}
	\zeta \ell^-_{D-1}& \equiv \tau(1-q^D)^2 \ell^-_D \qquad \pmod{\mathcal{L}} \label{eq:zeta.ell-(D-1)} \\
	 \zeta \ell^+_{D-1} &\equiv \tau q^D(1-q^D) \ell^-_D \qquad \pmod{\mathcal{L}} \label{eq:zeta.ell+(D-1)}\\
	 \zeta^{-1}\ell^+_{D-1} &\equiv \tau q^{D+1}(q^D-1) \ell^-_D + \tau(q^D-1)(q^{D+1}-1)\ell^+_D \qquad \pmod{\mathcal{L}}. \label{eq:zeta-1ell+(D-1)}
\end{align}
We shall now give the recurrence relations for $\ell^{\pm}_i$, $0 \leq i \leq D-1$.
\begin{theorem}\label{thm:3rdResult}
The following (i), (ii) hold.
\begin{itemize}
\item[(i)] For $0 \leq i \leq D-1$, $\zeta \ell^-_i$ and $\zeta^{-1} \ell^-_i$ are respectively given as linear combination with the following terms and coefficients.
\begin{align*}
	&\zeta \ell^-_i:  && \zeta^{-1} \ell^-_i:  \\
	& \begin{tabular}{c|c}
	term & coefficient \\
	\hline \hline \\ [-0.8em]
	$\ell^+_{i-1}$ & $\tau(q^i-q^D)(q^{N-D+1}-q^{i+1}+1) $\\[0.5em]
	$\ell^-_i$ & $\tau q^i(q^{N-D+1}-q^{i+1}+1)$ \\[0.5em]
	$\ell^+_i$ & $\tau (1-q^{i+1})(q^D-q^{i+1}+1)$ \\[0.5em]
	$\ell^-_{i+1}$ & $\tau (q^{i+1}-1)^2$ \\[0.5em]
	\end{tabular}, \quad
	&&
	\begin{tabular}{c|c}
	term & coefficient \\
	\hline \hline \\ [-0.8em]
	$\ell^-_{i-1}$ & $\tau(1-q^{i-D})(q^{N+1}-q^{D+i}) $\\[0.5em]
	$\ell^+_{i-1}$ & $\tau (q^D-q^{i})(q^{N-D+1}-q^i+1)$ \\[0.5em]
	$\ell^-_{i}$ & $\tau q^{i+1}(q^D-q^i+1)$ \\[0.5em]
	$\ell^+_{i}$ & $\tau (q^{i+1}-1)(q^D-q^i+1)$ \\[0.5em]
	\end{tabular}.
\end{align*}

\item[(ii)] For $0 \leq i \leq D-1$, $\zeta \ell^+_i$ and $\zeta^{-1} \ell^+_i$ are respectively given as linear combination with the following terms and coefficients.
\begin{align*}
	& \zeta\ell^+_i :  && \zeta^{-1}\ell^+_i : \\
	&\begin{tabular}{c|c}
	term & coefficient \\
	\hline \hline \\ [-0.8em]
	$\ell^+_{i-1}$ & $\tau q^{N+1}(q^{i-D}-1)(q^{D-N+i}-1)$\\[0.5em]
	$\ell^-_{i}$ & $\tau q^{N-D+1+i}(q^{D-N+i}-1)$ \\[0.5em]
	$\ell^+_{i}$ & $\tau q^{i+1}(q^D-q^{i+1}+1)$ \\[0.5em]
	$\ell^-_{i+1}$ & $\tau q^{i+1}(1-q^{i+1})$ \\[0.5em]
	\end{tabular}, \quad
	&&
	\begin{tabular}{c|c}
	term & coefficient \\
	\hline \hline \\ [-0.8em]
	$\ell^-_{i}$ & $\tau q^{N-D+2+i}(1-q^{D-N+i})$\\[0.5em]
	$\ell^+_{i}$ & $\tau q^{i+1}(q^{N-D+1}-q^{i+1}+1)$ \\[0.5em]
	$\ell^-_{i+1}$ & $\tau q^{i+2}(q^{i+1}-1)$ \\[0.5em]
	$\ell^+_{i+1}$ & $\tau (q^{i+1}-1)(q^{i+2}-1)$ \\[0.5em]
	\end{tabular}.
\end{align*} 
\end{itemize}
\end{theorem}
\begin{proof}
By Remark \ref{rmk:degree.ell}, the Laurent polynomials $\zeta\ell^\pm_i$, $\zeta^{-1}\ell^\pm_i$  belong to $\mathcal{L}$ except $\zeta \ell^-_{D-1}$, $\zeta \ell^+_{D-1}$, and $\zeta^{-1}\ell^+_{D-1}$.
By Lemma \ref{lem:actionX} and Proposition \ref{prop:C_i=ell_i(X)x}, the Laurent polynomials $\zeta\ell^\pm_i$, $\zeta^{-1}\ell^\pm_i$ belonging to $\mathcal{L}$ are given as linear combination as shown in the above tables.
For the remaining three cases, use \eqref{eq:zeta.ell-(D-1)}--\eqref{eq:zeta-1ell+(D-1)}. 
Then again, by Lemma \ref{lem:actionX} and Proposition \ref{prop:C_i=ell_i(X)x}, the desired result follows.
\end{proof}

We now discuss orthogonality relations for $\ell^\pm_i$.
We first find the eigenvalues of $\mathcal{X}$ on $\mathbf{W}$.
From \eqref{eq:min.poly(X)}, $\mu$ has $2D$ mutually distinct zeros 
\begin{equation}\label{eq:eigvaluesXonW}
	\lambda_i := 
	\begin{cases}
	\tau q^i, & i=0,1,\ldots, D-1, \\
	\tau^{-1} q^{i}, & i = -1, -2, \ldots, -D,
	\end{cases}
\end{equation}
and hence $\mathcal{X}$ is multiplicity-free on $\mathbf{W}$.
Next, we find eigenvectors of $\mathcal{X}$ corresponding to $\lambda_i$, $-D \leq i \leq D-1$.
Recall a $\Phi^*$-standard basis $\{E_i\hat{x}\}^D_{i=0}$ for $M\hat{x}$ and a $\Phi^{\perp*}$-standard basis $\{E_iu^\perp_0\}^{D-1}_{i=1}$ for $M\hat{x}^\perp$.
We consider the following ordered basis $\mathcal{B}$ for $\mathbf{W}$:
\begin{equation*}\label{eq:dual-basis.W}
	\mathcal{B} = \{E_0\hat{x}, E_1\hat{x}, E_1{u}^\perp_0, E_2\hat{x}, E_2{u}^\perp_0, \ldots,  E_{D-1}\hat{x}, E_{D-1}{u}^\perp_0, E_D\hat{x} \}.
\end{equation*}
Observe that $\mathcal{B}$ is orthogonal.
Recall the projection $\pi$ (resp. $\widetilde \pi$) from $\mathbf{W}$ onto $M\hat{x}$ (resp. $M\hat{C}$).

\begin{lemma}\label{lem:Mat.wtpi_B} 
The following (i), (ii) hold.
\begin{itemize}
\item[(i)] The matrix representing $\pi$ with respect to $\mathcal{B}$ is
\begin{equation*}
	\mathrm{blockdiag}\Big[ \pi(0), \pi(1), \ldots, \pi(D-1), \pi(D) \Big],
\end{equation*}
where $\pi(0) = \pi(D)=[1]$ and $\pi(i)=\mathrm{diag}(1,0)$ for $1 \leq i \leq D-1$.
\item[(ii)] The matrix representing $\widetilde \pi$ with respect to $\mathcal{B}$ is
\begin{equation*}
	\mathrm{blockdiag}\Big[\widetilde\pi(0), \widetilde\pi(1), \ldots, \widetilde\pi(D-1), \widetilde\pi(D)\Big],
\end{equation*}
where $\widetilde\pi(0) = [1]$, $\widetilde\pi(D)=[0]$, and $\widetilde\pi(i)$, $1 \leq i \leq D-1$, is a $2\times 2$ matrix
\begin{equation*}
	\begin{bmatrix}
	\dfrac{q^i(q^{D-i}-1)(q^{N-D-i+1}-1)}{(q^D-1)(q^{N-D+1}-1)} & \dfrac{q^{i-1}(q^i-1)(q^{D-i}-1)(q^{N-i+1}-1)(q^{N-D-i+1}-1)}{(q-1)(q^D-1)(q^{N-D+1}-1)} \\[1em]
	\dfrac{q(q-1)}{(q^{N-D+1}-1)(q^D-1)} & \dfrac{(q^i-1)(q^{N-i+1}-1)}{(q^{N-D+1}-1)(q^D-1)}
	\end{bmatrix}.
\end{equation*}
\end{itemize}
\end{lemma}
\begin{proof}
(i) Since $\pi.E_i\hat{x}=E_i\hat{x}$, $0 \leq i \leq D$, and $\pi(M\hat{x}^\perp)=0$, the result follows. \\
(ii) Since $E_0\hat{x} \in M\hat{C}$ and $E_D\hat{x} \in M\hat{C}^\perp$, it follows that  $\widetilde \pi.E_0\hat{x} = E_0\hat{x}$ and $\widetilde \pi.E_D\hat{x} = 0$.
Assume $1 \leq i \leq D-1$.
We now compute $\widetilde \pi.E_i\hat{x}$ and $\widetilde \pi E_i.u^\perp_0$.
Recall $\hat{C}_0 = \hat{C}^-_0 + \hat{C}^+_0$.
Eliminate $\hat{C}^+_0$ using \eqref{eq(1):C+-;Aix,upi} at $i=1$ to obtain
\begin{equation}\label{eq(1):pf:mat.rep.wt pi_B}
	\hat{C}_0 = \hat{C}^-_0 + \frac{q-1}{q^D-1}A\hat{x} + \frac{q}{q^D-1}u^\perp_0.
\end{equation}
In this equation, eliminate $\hat{C}_0$ using \eqref{eq(1):C+-;Ci,tilde upi} at $i=0$ and solve the result for $\widetilde u^\perp_0$ to obtain
\begin{equation}\label{eq(2):pf:mat.rep.wt pi_B}
	\widetilde u^\perp_0 = (q^{N-D}-1)\hat{C}^-_0 + \frac{(q-1)^2}{q(1-q^D)}A\hat{x} + \frac{1-q}{q^D-1}u^\perp_0.
\end{equation}
Apply $E_i$ to both sides of each equation of \eqref{eq(1):pf:mat.rep.wt pi_B}, \eqref{eq(2):pf:mat.rep.wt pi_B} and simplify the result using $E_iA=\theta_iE_i$.
Then, by recalling $\hat{x}=\hat{C}^-_0$ and $\hat{C}=\hat{C}_0$,  we have
\begin{align}
	& E_i\hat{C} = \frac{q^i(q^{D-i}-1)(q^{N-D-i+1}-1)}{(q-1)(q^D-1)}E_i\hat{x} + \frac{q}{q^D-1}E_iu^\perp_0,  \label{eq(3):pf:mat.rep.wt pi_B}\\
	& E_i\widetilde u^\perp_0 = \frac{(q^i-1)(q^{N-i+1}-1)}{q(q^D-1)}E_i\hat{x} + \frac{1-q}{q^D-1}E_iu^\perp_0. \label{eq(4):pf:mat.rep.wt pi_B}
\end{align}
Solving the system of equations \eqref{eq(3):pf:mat.rep.wt pi_B}, \eqref{eq(4):pf:mat.rep.wt pi_B} for $E_i\hat{x}$ and $E_i u^\perp_0$, we find
\begin{align}
	& E_i\hat{x} = \frac{q-1}{q^{N-D+1}-1}E_i\hat{C} + \frac{q}{q^{N-D+1}-1}E_i\widetilde u^\perp_0, \label{eq(5):pf:mat.rep.wt pi_B}\\
	& E_i u^\perp_0 = \frac{(q^i-1)(q^{N-i+1}-1)}{q(q^{N-D+1}-1)}E_i\hat{C} + \frac{q^D(q^{i-D}-1)(q^{N-D-i+1}-1)}{(q^{N-D+1}-1)(q-1)}E_i\widetilde u^\perp_0.\label{eq(6):pf:mat.rep.wt pi_B}
\end{align}
Apply $\widetilde \pi$ to both sides of each of equations \eqref{eq(5):pf:mat.rep.wt pi_B}, \eqref{eq(6):pf:mat.rep.wt pi_B} and eliminate $E_i\hat{C}$ using \eqref{eq(3):pf:mat.rep.wt pi_B}. 
Simplify the result to obtain $\widetilde \pi.E_i\hat{x}$ and $\widetilde \pi.E_iu^\perp_0$, which are given by linear combinations of $E_i\hat{x}$ and $E_iu^\perp_0$. 
The desired result follows.
\end{proof}

\noindent
In Theorem \ref{thm:2ndResult} we have shown  how the projection $\pi$ (resp. $\widetilde \pi$) is related to $\mathcal{T}'$ (resp. $\mathcal{T}$) on $\mathbf{W}$.
Using this result and Lemma \ref{lem:Mat.wtpi_B} we obtain the following lemma.
\begin{lemma}\label{lem:MatTT'_B}
The following (i), (ii) hold.
\begin{itemize}
	\item[(i)] The matrix representing $\mathcal{T}$ with respect to $\mathcal{B}$ is 
	\begin{equation*}
	\mathrm{blockdiag}\Big[\mathcal{T}(0), \mathcal{T}(1), \ldots, \mathcal{T}(D-1), \mathcal{T}(D)\Big],
	\end{equation*}
	where $\mathcal{T}(0) = [k]$, $\mathcal{T}(D)=[-k^{-1}]$, and $\mathcal{T}(i)$, $1 \leq i \leq D-1$, is a $2\times 2$ matrix
	\begin{equation*}
	\begin{bmatrix}
	\dfrac{k(q^{N+1}+q^D-q^i-q^{N-i+1})}{q^D-1} & \dfrac{-kq^{i-1}(q^i-1)(q^{D-i}-1)(q^{N-i+1}-1)(q^{N-D-i+1}-1)}{(q-1)(q^D-1)} \\[1em]
	\dfrac{-kq(q-1)}{q^D-1} & \dfrac{k(q^{N-i+1}+q^i-q^{N-D+1}-1)}{q^D-1}
	\end{bmatrix}.
	\end{equation*}
	\item[(ii)] The matrix representing $\mathcal{T'}$ with respect to $\mathcal{B}$ is
\begin{equation*}
	\mathrm{blockdiag}\Big[\mathcal{T'}(0), \mathcal{T'}(1), \ldots, \mathcal{T'}(D-1), \mathcal{T'}(D)\Big],
\end{equation*}
where $\mathcal{T'}(0) = \mathcal{T'}(D)=[k']$ and $\mathcal{T'}(i)=\mathrm{diag}(k',-k'^{-1})$ for $1 \leq i \leq D-1$.
\end{itemize}
\end{lemma}
\begin{proof}
Use Theorem \ref{thm:2ndResult} and Lemma \ref{lem:Mat.wtpi_B}. 
\end{proof}

\begin{lemma}\label{lem:MatX_B}
The matrix representing $\mathcal{X}$ with respect to $\mathcal{B}$ is
\begin{equation*}
	\mathrm{blockdiag}\Big[ \mathcal{X}(0), \mathcal{X}(1), \ldots, \mathcal{X}(D-1), \mathcal{X}(D) \Big],
\end{equation*}
where $\mathcal{X}(0) = [\tau]$, $\mathcal{X}(D) = [\tau^{-1} q^{-D}]$, and $\mathcal{X}(i)$, $1 \leq i \leq D-1$, is a $2\times 2$ matrix
\begin{equation}\label{lem:eq:MatX(i)}
	\begin{bmatrix}
	\dfrac{\tau(q^{N+1}-q^i+q^D-q^{N+1-i})}{q^D-1} & \dfrac{\tau(q^i-1)(q^{D-i}-1)(q^{N-i+1}-1)(q^{i-1}-q^{N-D})}{(q-1)(q^D-1)} \\[1em]
	\dfrac{\tau q^{D+1}(1-q)}{q^D-1} & \dfrac{\tau(q^{N+D+1-i}+q^{D+i}-q^{N+1}-q^D)}{q^D-1}
	\end{bmatrix}.
\end{equation}
\end{lemma}
\begin{proof}
Recall $\mathcal{X}=\mathcal{T}'\mathcal{T}$.
Use this and Lemma \ref{lem:MatTT'_B}. The result routinely follows.
\end{proof}

Note that for each $1\leq i \leq D-1$ the matrix $\mathcal{X}(i)$ of  \eqref{lem:eq:MatX(i)} has the eigenvalues $\lambda_{i}$ and $\lambda_{-i}$; cf \eqref{eq:eigvaluesXonW}.
Now we find eigenvectors of $\mathcal{X}$ associated with $\lambda_i$, $-D \leq i \leq D-1$.
Define
\begin{align}
	& \mathbf{y}_i := \frac{(q^{D-i}-1)(q^{N-i+1}-1)}{(q^D-1)(q^{N-2i+1}-1)} E_i\hat{x} + \frac{q^{D+1-i}(q-1)}{(q^D-1)(q^{N-2i+1}-1)}E_i u^\perp_0, \label{eq(1):eigvecX}\\
	& \mathbf{y}_{-i} := \frac{q^{D-i}(q^i-1)(q^{N-D-i+1}-1)}{(q^D-1)(q^{N-2i+1}-1)} E_i\hat{x} - \frac{q^{D-i+1}(q-1)}{(q^D-1)(q^{N-2i+1}-1)}E_iu^\perp_0, \label{eq(2):eigvecX}
\end{align}
for $1 \leq i \leq D-1$. Moreover, define
\begin{equation}\label{eq(3):eigvecX}
	\mathbf{y}_0 := E_0\hat{x}, \qquad \qquad \mathbf{y}_{-D} := E_D\hat{x}.
\end{equation}
Observe that (i) the vectors $\mathbf{y}_i$, $-D \leq i \leq D-1$, are real; (ii) $\mathbf{y}_i + \mathbf{y}_{-i} = E_i\hat{x}$, $1 \leq i \leq D-1$, so that
\begin{equation}\label{eq:sum(yi)=x}
	\sum^{D-1}_{i=-D} \mathbf{y}_i = 	\mathbf{y}_0 + \sum^{D-1}_{i=1} (\mathbf{y}_i + \mathbf{y}_{-i}) + \mathbf{y}_{-D}  = \sum^D_{i=0}E_i\hat{x} = \hat{x}.
\end{equation}
\begin{proposition}\label{prop:eigvectorX}
Let $-D \leq i \leq D-1$. 
On $\mathbf{W}$, the vector $\mathbf{y}_i$ is an eigenvector of $\mathcal{X}$ associated with the eigenvalue $\lambda_i$. 
Moreover, the vectors $\mathbf{y}_i$ form an eigenbasis of $\mathcal{X}$ for $\mathbf{W}$.
\end{proposition}
\begin{proof}
The first assertion routinely follows from \eqref{eq(1):eigvecX}, \eqref{eq(2):eigvecX}, \eqref{eq(3):eigvecX} and  Lemma \ref{lem:MatX_B}.
The second assertion immediately follows from that $\mathcal{B}$ is an basis for $\mathbf{W}$.
\end{proof}

Recall the Hermitian inner product $\langle \cdot, \cdot \rangle$ defined on $\mathbb{C}^X$.
Since the basis $\mathcal{B}$ is orthogonal on $\mathbf{W}$, we observe that 
\begin{equation}\label{iff:orthogonality}
	\langle \mathbf{y}_i, \mathbf{y}_j \rangle \ne 0 \quad \text{ if and only if } \quad j \in \{ i, -i\},  \qquad  -D\leq i,j \leq D-1.
\end{equation}
We compute explicitly the non-zero inner products. 
We first compute $\lVert E_i\hat{x}\rVert^2$, $0 \leq i \leq D$, and  $\lVert E_i u^\perp_0 \rVert^2$, $1 \leq i \leq D-1$. 
For $0 \leq i \leq D$, let $m_i$ denote the scalar as in \eqref{eq:scalar m(i)} associated with $\Phi$. Using \eqref{m_formula} and the parameter sequence of $\Phi$ in Proposition \ref{prop:LSonMx,Mxp}(I), we routinely find
\begin{equation}\label{eq:m;q-formula}
	m_i = \frac{q^i(1-q^{N-2i+1})(q^{i+1};q)_{D-i}}{(q^{N-D+1};q)_{D-i+1}}.
\end{equation}
For $0 \leq i \leq D-2$, let $m^\perp_i$ denote the scalar as in \eqref{eq:scalar m(i)}  associated with $\Phi^\perp$.
Using \eqref{m_formula} and the parameter sequence of $\Phi$ in Proposition \ref{prop:LSonMx,Mxp}(II), we routinely find
\begin{equation}\label{eq:mperp;q-formula}
	m^\perp_i = \frac{q^{N-1}(1-q^{N-D-i})(1-q^{D-i-1})(1-q^{2i-N+1})(q^{i+1};q)_{D-2-i}}{(q-1)(q^{N-D};q)_{D-i}}.
\end{equation}

\begin{lemma}\label{lem:norm.Eix} Both
\begin{itemize}
	\item[(i)] $\lVert E_i\hat{x} \rVert^2 = \dfrac{q^i(1-q^{N-2i+1})(q^{i+1};q)_{D-i}}{(q^{N-D+1};q)_{D-i+1}}, \quad  0 \leq i \leq D,$
	\item[(ii)] $\lVert E_iu^\perp_0 \rVert^2  = q^{2N-D-i}\dfrac{(1-q^{D-N+i-1})(1-q^{D-i})(1-q^{2i-N-1})(q^{i};q)_{D+1-i}}{(q-1)^2(q^{N-D+1};q)_{D-i}},  \quad 1 \leq i \leq D-1.$
\end{itemize}
\end{lemma}

\begin{proof}
(i) Since $E^*_0E_iE^*_0 = m_iE^*_0$, we have $\lVert E_i\hat{x} \rVert^2=\lVert E_iE^*_0\hat{x} \rVert^2= \langle \hat{x}, E^*_0E_iE^*_0\hat{x} \rangle=m_i$.
By this and \eqref{eq:m;q-formula}, the result follows.\\
(ii) Similarly to (i), we have $\lVert E_iu^\perp_0\rVert^2 = m^\perp_{i-1}\lVert u^\perp_0\rVert^2$, where $\lVert u^\perp_0\rVert^2 = \dfrac{(q^D-1)(q^{D-1}-1)(q^{n-D}-1)}{q-1}$ by \eqref{prop:eq:u-perp(i)} and Lemma \ref{cardinality:C+-}.
Using this and \eqref{eq:mperp;q-formula}, the result routinely follows.
\end{proof}

\begin{lemma}
Recall the eigenvectors $\mathbf{y}_i$, $-D \leq i \leq D-1$, for $\mathcal{X}$ on $\mathbf{W}$  from \eqref{eq(1):eigvecX}--\eqref{eq(3):eigvecX}.
For $1 \leq i \leq D-1$, we have
\begin{align}
	& \lVert \mathbf{y}_i \rVert^2 = \frac{ (q^{D-i}-1)(q^{N+1}+q^D-q^{D+i}-q^i)(q^i;q)_{D-i} }{(q^i-1)(q^{N-2i+1}-1)(q^{N-D+1};q)_{D-i}}, \label{lem:eq(1):norm}\\
	& \lVert \mathbf{y}_{-i} \rVert^2 = \frac{q^{D+N-2i+1}(q^{N-D-i+1}-1)(q^D-q^{D-N+i-1}-q^i+1)(q^i;q)_{D-i}}{(q^{N+1-i}-1)(q^{N+1-2i}-1)(q^{N-D+1};q)_{D-i}}, \label{lem:eq(2):norm}\\
	& \langle \mathbf{y}_i, \mathbf{y}_{-i} \rangle = \frac{q^D(1-q^{D-i})(1-q^{N-D-i+1})(q^i;q)_{D-i}}{(1-q^{N-2i+1})(q^{N-D+1};q)_{D-i}}.\label{lem:eq(3):norm}
\end{align}
Moreover,
\begin{equation}\label{lem:eq(4):norm}
	\lVert \mathbf{y}_0 \rVert^2 = \frac{(q;q)_D}{(q^{N-D+1};q)_D}, \qquad \qquad
	\lVert \mathbf{y}_{-D} \rVert^2 = \frac{q^D(q^{N-2D+1}-1)}{q^{N-D+1}-1}.
\end{equation}
\end{lemma}
\begin{proof}
Evaluate $\lVert \mathbf{y}_i \rVert^2$ using \eqref{eq(1):eigvecX} and Lemma \ref{lem:norm.Eix}(i) and simplify the result to get \eqref{lem:eq(1):norm}. 
Similarly, we obtain \eqref{lem:eq(2):norm}, \eqref{lem:eq(3):norm} using \eqref{eq(1):eigvecX}, \eqref{eq(2):eigvecX}.
Line \eqref{lem:eq(4):norm} follows from Lemma \ref{lem:norm.Eix}(i) at $i=0,D$.
\end{proof}

Recall the space $\mathcal{L}$ from \eqref{eq:space L} and let $f \in \mathcal{L}$. 
By \eqref{eq:sum(yi)=x} and Proposition \ref{prop:eigvectorX}  the action $f(\mathcal{X}).\hat{x}$ on $\mathbf{W}$ is given as
\begin{equation*}
	f(\mathcal{X}).\hat{x} = f(\mathcal{X}).\sum^{D-1}_{i=-D}\mathbf{y}_i = \sum^{D-1}_{i=-D}f(\lambda_i)\mathbf{y}_i.
\end{equation*}
Using this and \eqref{iff:orthogonality}, we find that for $f, g \in \mathcal{L}$,
\begin{equation*}\label{eq:innerprod_W}
	\langle f(\mathcal{X}).\hat{x}, g(\mathcal{X}).\hat{x} \rangle 
	=  \sum^{D-1}_{i=-D} f(\lambda_i) \overline{g(\lambda_i)} \lVert \mathbf{y}_i \rVert^2   + \sum^{D-1}_{i=1} \Big( f(\lambda_i)\overline{g(\lambda_{-i})} + f(\lambda_{-i})\overline{g(\lambda_i)} \Big) \langle \mathbf{y}_i, \mathbf{y}_{-i} \rangle.
\end{equation*}
In particular, for basis elements $\ell^\pm_i, \ell^\pm_j$ of $\mathcal{L}$, by using Proposition \ref{prop:C_i=ell_i(X)x}
\begin{equation}\label{eq:innerprod,ell,W}
	\langle \ell^\sigma_i (\mathcal{X}).\hat{x}, \ell^\nu_j(\mathcal{X}).\hat{x} \rangle 
	=  \langle \hat{C}^\sigma_i, \hat{C}^\nu_j \rangle
	= \delta_{i,j}\delta_{\sigma, \nu} |C^\sigma_i|,
\end{equation}
where $0 \leq i, j \leq D-1$ and $\sigma, \nu \in \{+,-\}$.
Motivated by these comments, we define the Hermitian form $\langle \cdot, \cdot \rangle_\mathcal{L}$ on $\mathcal{L}$ by 
\begin{equation}\label{eq:HermitianFormL}
	\langle f, g \rangle_{\mathcal{L}} :=  \sum^{D-1}_{i=-D} f(\lambda_i) \overline{g(\lambda_i)}  \omega_i   + \sum^{D-1}_{i=1} \Big( f(\lambda_i)\overline{g(\lambda_{-i})} + f(\lambda_{-i})\overline{g(\lambda_i)} \Big) \omega^\vee_i,
\end{equation}
where $f,g \in \mathcal{L}$ and the $\lambda_i$ are from \eqref{eq:eigvaluesXonW} and the non-zero (real) scalars $\omega_i$, $\omega^\vee_j$ are given by
\begin{equation*}
	\omega_i := \lVert \mathbf{y}_i \rVert^2, \quad -D \leq i \leq D-1, \qquad \omega^\vee_i := \langle \mathbf{y}_i, \mathbf{y}_{-i} \rangle, \quad 1 \leq i \leq D-1. 
\end{equation*}
We shall now give the orthogonality relations for $\ell^\pm_i$, $0 \leq i \leq D-1$.
\begin{theorem}\label{thm:4thResult}
Let $\langle \cdot , \cdot \rangle_{\mathcal{L}}$ be the Hermitian form as in \eqref{eq:HermitianFormL}.
For $0 \leq i, j \leq D-1$ and $\sigma, \nu \in \{+,-\}$, we have 
\begin{align*}
	\langle \ell^\sigma_i, \ell^\nu_j \rangle_{\mathcal{L}} 
	& = \delta_{i,j}\delta_{\sigma, \nu}|C^\sigma_i| \\
	& = 
	\begin{cases}
	\delta_{i,j}\delta_{\sigma, \nu} q^{i(i+1)}\displaystyle\prod^i_{h=1}  \dfrac{(q^{D-h}-1)(q^{N-D+1-h}-1)}{(q^h-1)^2}, & \text{ if }\sigma=-,\\
	\delta_{i,j}\delta_{\sigma, \nu} \dfrac{q^{(i+1)^2}(q^{N-D}-1)}{q-1}\displaystyle\prod^i_{h=1}  \dfrac{(q^{D-h}-1)(q^{N-D-h}-1)}{(q^h-1)(q^{h+1}-1)}, & \text{ if } \sigma=+.
	\end{cases}
\end{align*}
\end{theorem}
\begin{proof}
From \eqref{eq:innerprod,ell,W} and Lemma \ref{cardinality:C+-}, the result  follows.
\end{proof}

\section*{Acknowledgement}
The author thanks to Hajime Tanaka and Paul Terwilliger for giving this paper a close reading and offering valuable suggestions. 
The author also thanks the anonymous referee for careful reading and helpful comments.


\end{document}